\documentclass [10pt, leqno]{article}
\usepackage {graphicx}
\usepackage{wrapfig}
\usepackage{amsmath,amssymb, amsthm}
\usepackage[colorlinks,linkcolor=blue,citecolor=blue,urlcolor=blue, linktocpage=true,dvips]{hyperref}

\usepackage{xcolor}
\definecolor{light-gray}{gray}{0.45}

\numberwithin{equation}{section}

\newtheorem{theorem}{Theorem}[section]

\newtheorem{lemma}[theorem]{Lemma}
\newtheorem{corollary}[theorem]{Corollary}
\newtheorem*{theorem*}{Theorem}

\newtheorem{teorA}{Theorem}

\newtheorem{remark}[theorem]{Remark}

\def\nfracj{\Big\lfloor\frac{n}{j}\Big\rfloor}

\def\unr{U^{(n)}_r}
\def\unrk{U^{(n)}_r(k)}
\def\wnr{W^{(n)}_r}
\def\wnrk{W^{(n)}_r(k)}
\def\P{\mathbf{P}}

\def\Q{\mathbf{Q}}
\def\E{\mathbb{E}}
\def\V{\mathbb{V}}
\def\unogrande{\text{\large\bf 1}}

\def\lcm{\text{\rm lcm}}

\begin{document}

\title{Asymptotic normality and  greatest common divisors}
 \author{Jos\'{e} L. Fern\'{a}ndez and  Pablo Fern\'{a}ndez}

 \date{\today}

 \renewcommand{\thefootnote}{\fnsymbol{footnote}}
 \footnotetext{\noindent\emph{2010 Mathematics Subject Classification}: 60F05, 62E20, 11N37, 62H10.}
\footnotetext{\noindent\emph{Keywords}: Gcd, asymptotic normality, U-statistics, locally dependent random variables, Euler, Jordan and Pillai functions, strong law.}
\renewcommand{\thefootnote}{\arabic{footnote}}

 \maketitle

 \begin{abstract}
 {We report on some statistical regularity properties of  greatest common divisors: for large random samples of integers, the number of coprime pairs and the average of the gcd's of those pairs are approximately normal, while the maximum of those gcd's (appropriately normalized) follows approximately a Fr\'echet distribution approximately. We also consider $r$-tuples instead of pairs, and moments other than the average.} \end{abstract}

\section{Introduction}

In this paper we report on some  statistical regularities  of the greatest common divisors of  random pairs, or more generally, $r$-tuples, drawn from  large samples of integers.

For  any given integer $n \ge 1$, let us denote by $X^{(n)}_1, X^{(n)}_2, \ldots$ a sequence of independent random variables \textit{uniformly distributed} in $\{1, \ldots, n\}$  and defined on a certain given probability space endowed with a probability $\P$.

%For a concrete realization we may take as probability space the unit interval with Lebesgue measure and Borel %$\sigma$-algebra, and for each $\omega \in [0,1]$ and integers $n,j \ge 1$ we could define  $X^{(n)}_j(\omega)$ as 1 %plus the $j$-th digit of the expansion in base $n$ of $\omega$.

The distribution of  $\gcd(X^{(n)}_1, X^{(n)}_2)$, the $\gcd$ of a random pair, is given by
$$
\P\big(\gcd(X^{(n)}_1, X^{(n)}_2)=k\big)=\frac{1}{n^2}\sum_{j \le {n}/{k}} \mu(j)\Big\lfloor\frac{n}{jk}\Big\rfloor^2\, ,
$$
for  $ 1 \le k\le n$.
Asymptotically, as $n \to \infty$, one has
$$
\lim_{n \to \infty} \P\big(\gcd(X^{(n)}_1, X^{(n)}_2)=k\big)=\frac{1}{\zeta(2)} \frac{1}{k^2}\, ,
$$
which, in particular, for $k=1$, is the classical result of Dirichlet (see, for instance, \cite{HW}, Theorem 332) that
\begin{equation}
\label{eq:dirichlet}
\lim_{n \to \infty}\P\big(\gcd(X^{(n)}_1, X^{(n)}_2)=1\big)=\frac{1}{\zeta(2)}\, .
\end{equation}
For the mean and the variance of $\gcd(X^{(n)}_1, X^{(n)}_2)$ one has the asymptotic results
\begin{equation*}\label{eq:esperanza de gcd pares}
\E\big(\gcd\big(X^{(n)}_1, X^{(n)}_2\big)\big)\sim \frac{1}{\zeta(2)} \ln(n) \quad \text{and} \quad
\V\big(\gcd\big(X^{(n)}_1, X^{(n)}_2\big)\big) \sim \Big[\frac{1}{3} \Big(\frac{2 \zeta(2)}{\zeta(3)}-1\Big)\Big] \, n\, .
\end{equation*}
as $n \to \infty$. We refer to E. Ces\`{a}ro \cite{Ce1885}, E. Cohen \cite{Cohen}, and
P. Diaconis and P. Erd\H os \cite{DE2004} for some further details and references. See also Section \ref{section:marginal probs and expec} of this paper.

\smallskip

Fix $n \ge 1$. For each integer $m \ge 2$, consider  the random variable
$$
\mathcal{C}^{(n)}_m=\sum_{1\le i<j\le m} \text{\large\bf 1}_{\gcd(X^{(n)}_i,X^{(n)}_j)=1}\, ,
$$
which counts the number of \textit{coprime pairs} in a random sample of length $m$ drawn from $\{1,\dots,n\}$. Observe that~$\mathcal{C}^{(n)}_m$ does not exceed $\binom{m}{2}$ and attains that maximum value precisely when the whole sample $\big(X^{(n)}_1, X^{(n)}_2, \ldots, X^{(n)}_m\big)$ is pairwise coprime. The formula
\begin{align*}
\lim_{n \to \infty}\P\big(\big(X^{(n)}_1, \ldots, X^{(n)}_m\big) \ \text{pairwise coprime} \big)
&=\lim_{n \to \infty}\P\Big(\mathcal{C}^{(n)}_m=\binom{m}{2} \Big)
\\
&
=\prod_{p} \Big(\Big(1-\frac{1}{p}\Big)^{m-1}\Big(1+\frac{m-1}{p}\Big)\Big):= T_m
\end{align*}
was proved by L. Toth, \cite{To2004}, and also by J. Cai and E. Bach, \cite{CB2001}.
In the case $m=2$, this limit probability  reduces to the classical result of Dirichlet mentioned above, $T_2={1}/{\zeta(2)}$.

As the  size $m$ of the sample tends to $\infty$, the probability $T_m$ of pairwise coprimality tends to 0, see \cite{To2004}, and also \cite{Hwang}. This is to  be compared with the extension of Dirichlet's Theorem, see Section \ref{section:euler and pillai} for references,  that for each $m \ge 2$,
$$
\lim_{n \to \infty}\P\big(\big(X^{(n)}_1, \ldots, X^{(n)}_m\big) \ \text{coprime} \big)=\lim_{n \to \infty}\P\big(\gcd\big(X^{(n)}_1, \ldots, X^{(n)}_m\big) =1 \big)=\frac{1}{\zeta(m)}\, .
$$
Now, the  probability ${1}/{\zeta(m)}$ of just (mutual) coprimality tends to 1, as the sample size $m$ tends to $\infty$.

The \textit{exact} distribution of $\mathcal{C}^{(n)}_m$, for sample size $m$ given and $n$ fixed, is combinatorially involved; see J. Hu, \cite{HU}, for an interesting approach.

\medskip

In this paper we prove that $\mathcal{C}^{(n)}_m$ is \textit{asymptotically normal} as $m$ tends to $\infty$ when $n$ is fixed and, more generally, when $n$ is allowed to vary with $m$, with the only restriction that~$n \ge 2$.

\begin{teorA}\label{theorA}
For each fixed $n \ge 2$,
$$
\frac{\mathcal{C}^{(n)}_m-\E(\mathcal{C}^{(n)}_m)}{\sqrt{\V(\mathcal{C}^{(n)}_m})} \overset{\text{d}}{\longrightarrow} \mathcal{N}\,, \quad \text{as} \ \ m \to \infty\, .
$$
More generally, the conclusion holds with $n$ replaced by any sequence $n_m$ of integers $n_m \ge 2$.
\end{teorA}
(This is Theorem  \ref{th:TLC para indicadores} in Section \ref{section:statistics_gcd_pairs}). By $\overset{\text{\it d}}{\longrightarrow}$ we mean convergence in distribution; $\mathcal{N}$ represents a standard normal variable.

The counter $\mathcal{C}^{(n)}_m$ is a sum of $\binom{m}{2}$ Bernoulli variables with common probability of success, but, of course, they are not independent.

The analysis of $\mathcal{C}^{(n)}_m$ could be framed into, at least, two different approaches. On the one hand, for fixed $n$, we could consider $\mathcal{C}^{(n)}_m$, or rather $\mathcal{C}^{(n)}_m/\binom{m}{2}$, as a sequence of $U$-statistics associated to the symmetric kernel $\gcd(x,y)$ and apply some general asymptotic results of W. Hoeffding, \cite{Hoeffding}. Alternatively, we could consider the collection of random variables $\gcd(X^{(n)}_i, X^{(n)}_j), \ 1 \le i <j \le m$, as a family of \textit{locally dependent} and identically distributed variables, and apply some general limit theorems for the sum of such a family, like those of  S. Janson, \cite{Ja1988}, or P. Baldi and Y. Rinnot, \cite{BR1} and \cite{BR2}. This second approach appears to be more flexible, particularly when $n$ is allowed  to vary with $m$; it is the one we shall follow.

Both approaches depend on appropriate \textit{estimates of covariances of pairs} of variables
$\big(\unogrande_{\gcd(X^{(n)}_i, X^{(n)}_j)=1}, \unogrande_{\gcd(X^{(n)}_k, X^{(n)}_l)=1}\big)$, of number theoretical nature, which we discuss in Sections~\ref{section:euler and pillai} and~\ref{section:marginal probs and expec}.

\medskip

We  also consider  some other natural $U$-statistics like the sum of $\gcd$ of pairs from the sample,
$$
\mathcal{Z}^{(n)}_m=\sum_{1\le i<j\le m} \gcd(X^{(n)}_i,X^{(n)}_j)\,,
$$
instead of counting coprime pairs. We have:
\begin{teorA}\label{theorB}
For each fixed $n \ge 2$,
$$
\frac{\mathcal{Z}^{(n)}_m-\E(\mathcal{Z}^{(n)}_m)}{\sqrt{\V(\mathcal{Z}^{(n)}_m})} \overset{\text{d}}{\longrightarrow} \mathcal{N}\, , \quad \text{as} \ \ m \to \infty\, .
$$
More generally, the conclusion holds with $n$ replaced by any sequence $n_m$ of integers $n_m \ge 2$ which verify $n_m \le m^\beta$, for $\beta <1/2$.
\end{teorA}
(This is Theorem  \ref{th:TLC para gcd} in Section \ref{section:statistics_gcd_pairs}). Notice that, in contrast to Theorem \ref{theorA},  it is now required that the size of the sample space $n_m$ does not grow too fast as compared with the sample size $m$. It would be interesting to know whether this is really necessary and not just a restriction of the method of proof.

\medskip

We  consider also, in the opposite end, the random variables
$$
\mathcal{M}^{(n)}_m=\max_{1\le i<j\le m} \big\{\gcd(X^{(n)}_i,X^{(n)}_j)\big\}.
$$
or, rather, their normalized version $\widetilde{\mathcal{M}}^{(n)}_m=\binom{m}{2}^{-1} \mathcal{M}^{(n)}_m$.
In \cite{DP2011}, Darling and Pyle obtained some interesting asymptotic results about $\widetilde{\mathcal{M}}^{(n)}_m$, and asked whether it has a limit, in distribution, as $m\to\infty$. That this is the case is the content of:
\begin{teorA}\label{theorC}
Let $m^\beta\le n \le e^{m^{\gamma}}$, for some $\beta>2$ and $\gamma <{1}/{3}$.
Then, for any $t>0$,
$$
\lim_{m\to\infty}\P\big(\widetilde{\mathcal{M}}^{(n)}_m \le t\big)= \exp\Big(-\frac{1}{t\zeta(2)}\Big),
$$
so that $\widetilde{\mathcal{M}}^{(n)}_m$ tends, in distribution,  as $m\to\infty$, to the Fr\'{e}chet distribution with shape parameter~$1$ and scale parameter $1/\zeta(2)$.
\end{teorA}

(This is Theorem  \ref{th:distribution of max of gcd} in Section \ref{section:statistics_gcd_pairs}). Our derivation of Theorem \ref{theorC} is based on a classical result of Brown and  Silverman (see~\cite{BS1979},~\cite{SB1978}) on Poisson approximation of $U$-statistics.

\smallskip

These theorems, \ref{theorA}, \ref{theorB}, and \ref{theorC}, have corresponding counterparts for $\gcd$ of $r$-tuples, instead of just pairs, or for higher moments of $\gcd$ instead of just first moments, which we discuss in  Sections~\ref{section:statistics_r_tuples}~and~\ref{section:statistics_higher_moments}.

\smallskip

The paper is organized as follows.
Section \ref{section:euler and pillai} contains results about Euler's $\varphi$ and Pillai's~$P$ function which are needed later.
Section \ref{section:marginal probs and expec} derives some estimates of marginal probabilities and expectations, and of the appropriate covariances.
Section  \ref{section:statistics_gcd_pairs} contains the proofs of Theorems \ref{theorA}, \ref{theorB}, and \ref{theorC}.
Section \ref{section:statistics_r_tuples} considers the extension of those results to $r$-tuples, while Section \ref{section:statistics_higher_moments} discusses the extension to higher moments.
Finally, Section \ref{section:strong_law} discusses a strong law for $\gcd$.

\medskip

\noindent \textbf{Some notation:}

At a number of places we shall have products indexed by prime numbers: $\prod_{p}$ means product running over all primes $p$, while $\prod_{p \le k}, \prod_{p|k}$ are products running over primes which are less than or equal to $k$, and over primes which divide $k$, respectively.

We denote by $I_s$ the arithmetic function $I_s(j)=j^s$ and simply write $I$ for $I_1$. With $\delta_k$ we denote the arithmetic  function $\delta_k(j)=1$ if $j=k$, and $\delta_k(j)=0$ otherwise. The number of divisors of an integer  $j\ge 1 $ is denoted by $\tau(j)$.
For any positive real number $x$, we denote by $\{x\}$ its fractional part: $\{x\}=x-\lfloor x \rfloor$. The M\"obius function is $\mu$, and $\ast$ denotes Dirichlet convolution.
%For any arithmetic function $F$, its  summatory function $\mathcal{S}F$ is the  function defined on $(0, +\infty)$ by
%$$
%\mathcal{S}F(x)=\sum_{k \le x} F(k) \, .
%$$
For two sequences of positive numbers $(a_n)$ and $(b_n)$, by $a_n \sim b_n$ as $n \to \infty$, we mean that $\lim_{n \to \infty} {a_n}/{b_n}=1$.

\section{Euler's $\varphi$, Pillai's $P$ function, and extensions}\label{section:euler and pillai}

We collect in this section a number of identities and estimates involving Euler's $\varphi$ function, Pillai's $P$ function, and  their corresponding  $s$-dimensional versions $\varphi_s$ (Jordan's totient functions) and $P_s$.

\subsection{Euler's and Jordan's function} Euler's $\varphi$ function,
$$
\varphi(k)=\sum_{j=1}^k \unogrande_{\gcd(j,k)=1}, \quad \text{ for each} \ k \ge 1\,,
$$
satisfies the identity $\varphi=\mu \ast I$, and  verifies that
$$
\frac{\varphi(k)}{k}=\prod_{p |k} \Big(1-\frac{1}{p}\Big), \quad \text{ for each} \ k \ge 1\,.
$$
Observe that  $\varphi(k) \le k$, for every integer $k \ge 1$.

\subsubsection{A double series involving $\varphi$ and $\gcd$}
For every $t >1$,  define
\begin{equation}
\label{eq:def of M(t)}
M(t):=\sum_{i,j=1}^\infty \frac{\varphi(i)}{i^{1+t}}\,\frac{\varphi(j)}{j^{1+t}} \,\gcd(i,j)\,.
\end{equation}
The following identity shall prove  useful:
\begin{lemma}\label{lemma:identity_double_sum_varphi_gcd}
For every  $t >1$,
\begin{align}\label{eq:identity_double_sum_varphi_gcd}
M(t)&=\zeta(2t-1) \prod_{p} \Big(1+\frac{2}{p^{t}}-\frac{2}{p^{t+1}}-\frac{1}{p^{2t+1}}\Big)
\\\nonumber
&=\zeta(2t-1) \zeta(t)^2 \prod_{p}\Big(1-\frac{2}{p^{t+1}}-\frac{3}{p^{2t}}+\frac{3}{p^{2t+1}}+\frac{2}{p^{3t}}-\frac{1}{p^{4t+1}}\Big)\, .
\end{align}
\end{lemma}

Observe, in particular, that $M(t) < + \infty$ for every $t >1$.
The second product expression for $M(t)$ in \eqref{eq:identity_double_sum_varphi_gcd} will be most convenient so as to apply some Tauberian theorem, see Corollary \ref{cor:partial_sums_identity_double_sum_varphi_gcd}.

The proof of Lemma \ref{lemma:identity_double_sum_varphi_gcd} uses  the so-called \textit{zeta} (probability) distributions on $\mathbb{N}$: for each real $t  >1$,  the zeta distribution $\Q_t$ on $\mathbb{N}$ is given by
$$
\Q_t(j)=\frac{1}{\zeta(t)}\frac{1}{j^t}\, , \quad \text{for each integer} \  j \ge 1\, .
$$
For every prime number $p$, the random variable $\alpha_p$ on $\mathbb{N}$ assigns to each integer $j \ge 1$ the largest exponent $\alpha\ge 0$ so that $p^{\alpha}|j$ (thus  $p^{\alpha(j)}|j$, but $p^{\alpha_p(j)+1} \nmid j$); in particular, $\{\alpha_p>0\}$ is the event  \lq\lq divisible by $p$\rq\rq, and, besides,
$$
j=\prod_{p} p^{\alpha_p(j)}\, , \quad \text{ for each integer} \ j \ge 1\, .
$$

With respect to $\Q_t$, the variables $\{\alpha_p\}_{p}$ are mutually independent, and, moreover,   each~$\alpha_p$ is distributed as a geometric random variable on $\{0, 1, 2, \ldots\}$ with success probability~$1-{1}/{p^t}$:
$$
\Q_t(\alpha_p=k)=\Big(1-\frac{1}{p^t}\Big)\frac{1}{p^{tk}}\, , \quad \text{ for each integer} \ k \ge 0\, .
$$
See Golomb \cite{Go},  Diaconis \cite{D1976}, Kingman \cite{Kingman}, and, particularly, Lloyd \cite{Lloyd}.

\begin{proof}[Proof of Lemma {\upshape\ref{lemma:identity_double_sum_varphi_gcd}}]
We first observe that the second infinite product expression follows from the first one and the Euler product expansion for $\zeta(t)=\prod_{p}\frac{1}{1-p^{-t}}$; so
that we just verify the first one.

We denote by $\Q_t^2$ the product probability $\Q_t \times \Q_t$ on $\mathbb{N}^2$ and write $\E_{\Q_t^2}$ for the corresponding expectations. Consider the variable $G$ on $\mathbb{N}^2$ given by
$$
(i,j) \in \mathbb{N}^2 \ \mapsto\ G(i,j)=\frac{\varphi(i)}{i}\frac{\varphi(j)}{j} \gcd(i,j)\, .
$$
Observe that, for $t >1$,
$$
\E_{\Q^2_t}(G)=\frac{1}{\zeta(t)^2}\sum_{i,j=1}^{\infty}\frac{\varphi(i)}{i^{1+t}}\frac{\varphi(j)}{j^{1+t}} \gcd(i,j)\, .
$$

We introduce the auxiliary arithmetic function $h$ given by $h(j)=1$, if $j\ge 1$, and $h(0)=0$, so that we may write $\varphi$ and $\gcd$ in terms of the variables $\alpha_p$ as
$$
\frac{\varphi(j)}{j}=\prod_{p \mid j} \Big(1-\frac{1}{p}\Big)=\prod_{p}\Big(1-\frac{h(\alpha_p(j))}{p}\Big)
\quad \text{\rm and} \quad \gcd(i,j)=\prod_{p} p^{\min{(\alpha_p(i), \alpha_p(j))}}\, ,
$$
and then $G$ itself as
$$
G(i,j)=\prod_{p} \Big(1-\frac{h(\alpha_p(i))}{p}\Big)\Big(1-\frac{h(\alpha_p(j))}{p}\Big)p^{\min{(\alpha_p(i), \alpha_p(j))}}\, ,
$$
which is an infinite product of mutually independent random variables.

Now, for each fixed prime $p$, we have that
$$
\begin{aligned}
\E_{\Q^2_t} \Big[\Big(1-\frac{h(\alpha_p(i))}{p}\Big)&\Big(1-\frac{h(\alpha_p(j))}{p}\Big)p^{\min{(\alpha_p(i), \alpha_p(j))}}\Big]
\\
&=\sum_{k,l=0}^{\infty}\Big(1-\frac{h(k)}{p}\Big)\Big(1-\frac{h(l)}{p}\Big)p^{\min{(k,l)}}\Big(1-\frac{1}{p^t}\Big)^2 \frac{1}{p^{tk}}\frac{1}{p^{tl}}\, .
\end{aligned}
$$
Split the range of the double sum into  $\{k=0,l=0\}$, $\{k=0,l>0\}$, $\{k>0,l=0\}$ and $\{k>0,l>0\}$, sum several geometric series and simplify to get the compact expression:
\begin{align*}
\E_{\Q^2_t} \Big[\Big(1-\frac{h(\alpha_p(i))}{p}\Big)&\Big(1-\frac{h(\alpha_p(j))}{p}\Big)p^{\min{(\alpha_p(i), \alpha_p(j))}}\Big]
\\
=& \frac{\big(1-{1}/{p^{t}}\big)^2}{\big(1-{1}/{p^{2t-1}}\big)}\ \Big(1+\frac{2}{p^t}-\frac{2}{p^{t+1}}-\frac{1}{p^{2t+1}}\Big)\, .
\end{align*}
Now, since the $\alpha_p$'s are mutually independent, we may write, at least formally, that
$$
\E_{\Q^2_t}(G)=\frac{\zeta(2t-1)}{\zeta(t)^2}\prod_{p}\Big(1+\frac{2}{p^t}-\frac{2}{p^{t+1}}-\frac{1}{p^{2t+1}}\Big)\, ,
$$
to obtain the desired result.

\smallskip

To justify the formal step, denote  $H(i,j)=\gcd(i,j)=\prod_{p} p^{\min{(\alpha_p(i), \alpha_p(j))}}$.
%From monotone convergence, independence and the fact that $\min{(\alpha_p(i),\alpha_p(j))}$ is again
%a geometric variable on $\{0,1,2, \ldots, \}$ but with probability of success $1-\frac{1}{p^{2s}}$, one deduces that
Observe that
$$
\E_{\Q^2_t}(H)=\frac{1}{\zeta(t)^2}\sum_{i,j=1}^{\infty} \frac{\gcd(i,j)}{i^t j^t}=\frac{\zeta(2t-1)}{\zeta(2t)} < + \infty\, .
$$
The last identity follows from the following elementary argument with the M\"{o}bius function: for any arithmetical function $f$
\begin{align}
\nonumber
\sum_{\substack{1\le x_1,\dots, x_r\le n\\ \gcd(x_1,\dots, x_r)=1}} f(x_1,\dots, x_r)&=\sum_{k=1}^n \mu(k) \sum_{\substack{1\le x_1,\dots, x_r\le n\\ k| x_1,\dots, k|x_r}} f(x_1,\dots, x_r)
\\ \label{eq:use of mu}
&=\sum_{k=1}^n \mu(k) \sum_{1\le y_1,\dots, y_r\le  n/k} f(ky_1,\dots, ky_r),
\end{align}
Using \eqref{eq:use of mu}, we can write
\begin{align}\nonumber
\sum_{i,j=1}^{\infty} \frac{\gcd(i,j)}{i^t j^t}&=\sum_{d=1}^\infty d \sum_{\gcd(i,j)=d} \frac{1}{i^t j^t}=
\sum_{d=1}^\infty \frac{1}{d^{2t-1}} \sum_{\gcd(a,b)=1} \frac{1}{a^t b^t}
\\ \label{eq:sum gcd and zeta}
&=
\sum_{d=1}^\infty \frac{1}{d^{2t-1}} \sum_{k=1}^\infty \frac{\mu(k)}{k^{2t}}\sum_{a,b=1}^\infty \frac{1}{a^t b^t}=\frac{\zeta(2t-1)\, \zeta(t)^2}{\zeta(2t)}.
\end{align}

For every integer $N \ge 1$, define the partial product $G_N$  as
$$
G_N(i,j)=\prod_{p \le N} \Big(1-\frac{h(\alpha_p(i))}{p}\Big)\Big(1-\frac{h(\alpha_p(j))}{p}\Big)p^{\min{(\alpha_p(i), \alpha_p(j))}}
$$
Now, $G_N (i,j) \le H(i,j)$, and $\lim_{N \to \infty} G_N(i,j)=G(i,j)$, for any integers $i,j \ge 1$ and so,  by dominated convergence, we deduce
$$
\lim_{N \to \infty}\E_{\Q^2_t}(G_N)=\E_{\Q^2_t}(G)\, .
$$
And, finally, since $G_N$ is a finite product of independent variables, we have
$$
\E_{\Q^2_t}(G_N)=\prod_{p \le N}\frac{\big(1-{1}/{p^{t}}\big)^2}{\big(1-{1}/{p^{2t-1}}\big)}\ \Big(1+\frac{2}{p^t}-\frac{2}{p^{t+1}}-\frac{1}{p^{2t+1}}\Big)
$$
and the proof is completed.\end{proof}

The double sum which would correspond to $t=1$ is infinite:
$$
\sum_{i,j=1}^\infty \frac{\varphi(i)}{i^{2}}\,\frac{\varphi(j)}{j^{2}} \,\gcd(i,j)=+\infty\,;
$$
the following corollary gives a suitable estimate for its rate of convergence to $\infty$.

\begin{corollary}
\label{cor:partial_sums_identity_double_sum_varphi_gcd}As $N \to \infty$,
$$
\sum_{i\cdot j \le N}\frac{\varphi(i)}{i^{2}}\,\frac{\varphi(j)}{j^{2}} \,\gcd(i,j) \sim \Delta \ln(N)^3\,,
$$
where $\Delta$ is the number
$$
\Delta=\frac{1}{12} \prod_{p} \Big(1-\frac{5}{p^2}+\frac{5}{p^3}-\frac{1}{p^5}\Big)\approx 0,01186\, .
$$
$($The summation above is over the set of integers $i,j \ge 1$ whose product $i\cdot j \le N$.{\upshape)}
In particular,
$$
\liminf_{N \to \infty} \frac{1}{\ln(N)^3} \sum_{\lcm(i,j) \le N}\frac{\varphi(i)}{i^{2}}\,\frac{\varphi(j)}{j^{2}} \,\gcd(i,j) \ge \Delta\, .
$$
\end{corollary}

\medskip

In the proof of Corollary \ref{cor:partial_sums_identity_double_sum_varphi_gcd} we shall resort to (a particular case of) the powerful Delange's Tauberian Theorem, which we may write as follows:
\begin{theorem}
[Delange, \cite{Delange}, Th\'eor\`em 1]\label{th:delange}
Let $A(z):=\sum_{k=1}^\infty \frac{a_k}{k^z}$ be a Dirichlet series with nonnegative coefficients which has abscissa of convergence $\rho >0$ and is holomorphic on the whole axis $\Re(z)=\rho$ except at the point $s=\rho$.

Assume that for two functions $F(z)$ and $G(z)$, holomorphic in $\Re(z) \ge \rho$,  and for some real $\beta>0$ we have
\begin{equation}
\label{eq:delange1}
A(z)=\frac{F(z)}{(z-\rho)^\beta}+G(z)\, , \quad \text{for $\Re(z)>\rho,$}
\end{equation}
and $F(\rho)\neq 0$. Then, as $n \to \infty$,
\begin{equation}
\label{eq:delange2}
\sum_{k=1}^n a_k \sim \frac{F(\rho)}{\rho \,\Gamma(\beta)}\ n^\rho \ \big(\ln(N)\big)^{\beta-1}\,.
\end{equation}
\end{theorem}
For non integer $\beta$, the power $(z-\rho)^\beta$ in \eqref{eq:delange1} means    its principal branch.

\begin{proof}[Proof of Corollary {\upshape \ref{cor:partial_sums_identity_double_sum_varphi_gcd}}]
Observe first that the asymptotic  comparison closing  the statement of the corollary follows simply from the fact that  $\text{\rm{lcm}}(i,j)\le i \cdot j $.

\smallskip
Denote by $B(z)$ the (holomorphic and nonvanishing for $\Re(z) >{1}/{2}$) function
$$
%s \in \mathbb{C}, \Re(s) >\frac{1}{2} \mapsto%
B(z)=\prod_{p}\Big(1-\frac{2}{p^{z+1}}-\frac{3}{p^{2z}}+\frac{3}{p^{2z+1}}+\frac{2}{p^{3z}}-\frac{1}{p^{4z+1}}\Big),
$$
Notice that $B(1)= \prod_{p} (1-{5}/{p^2}+{5}/{p^3}-{1}/{p^5})=12\,\Delta$. Also, denote by $C$ the entire function $C(z)=(z-1)\zeta(z)$, for $z \in \mathbb{C}$, and observe that $C(1)=1$.

\smallskip
Extend the function $M$ given in \eqref{eq:def of M(t)} to a  holomorphic function in $\Re(z)>1$:
$$
M(z)=\sum_{i,j=1}^\infty \frac{\varphi(i)}{i^{1+z}}\frac{\varphi(j)}{j^{1+z}} \gcd(i,j)=\zeta(2z-1)\zeta(z)^2 B(z)\, .
$$
For each integer $k\ge 1$, define the \textit{positive} coefficient
$$
a_k=\sum_{i \cdot j=k}\frac{\varphi(i)}{i}\frac{\varphi(j)}{j} \gcd(i,j)
$$
to express $M$ as a Dirichlet series
$$
M(z)=\sum_{k=1}^{\infty} \frac{a_k}{k^z}\,,\quad\Re(z)>1.
$$

For $\Re(z) >1$ we may write
$$
M(z)=\frac{1}{(z-1)^3} \ \Big[\frac{1}{2}\,C(2z-1)\, C(z)^2\, B(z)\Big]\, .
$$
The function $F(z)=\frac{1}{2}C(2z-1) C(z)^2 B(z)$ is holomorphic for $\Re(z) >{1}/{2}$, and $F(1)=B(1)/2$. Delange's Tauberian Theorem (with $\rho=1$, $\beta=3$, $F$ as above and $G\equiv0$) gives then that
$$
\sum_{k=1}^n a_k \sim \frac{F(1)}{\Gamma(3)}  \,n \ln(n)^2=\frac{B(1)}{4}\,n \ln(n)^2\, , \quad \text{as} \ n \to \infty\, .
$$

From summation by parts, we finally deduce that
$$
\sum_{k=1}^n \frac{a_k}{k}\sim \frac{B(1)}{12} \ln(n)^3=\Delta \ln(n)^3\, , \quad \text{as} \ n \to \infty\, ,
$$
and, therefore, as desired, that
$$
\sum_{i \cdot j \le n} \frac{\varphi(i)}{i^2}\frac{\varphi(j)}{j^2}\gcd(i,j) \sim \Delta  \ln(n)^3\, ,  \quad \text{as} \ n \to \infty\, .\qquad \mbox{\qedhere}
$$
\end{proof}

\subsubsection{Jordan's functions} For each \textit{integer} $s\ge 1$, the ($s$-)\textit{Jordan totient function}, denoted here by  $\varphi_s$, is given by the convolution
$$\varphi_s=\mu\ast I_s\, .$$
For each integer $k \ge 1$, the function $\varphi_s$ counts the number of $s$-tuples of integers $(k_1, \ldots, k_s)$ with $1 \le k_1,\dots,k_s \le k$, such that $\gcd(k_1, \ldots, k_s, k)=1$.
Of course, $\varphi_1=\varphi$. Observe that
$$
\varphi_s(k)=k^s \sum_{j|k} \frac{\mu(j)}{j^s}=k^s \prod_{p|k} \Big(1-\frac{1}{p^s}\Big)\, , \quad \text{for each integer} \ k \ge 1\, .
$$
Notice also that $\varphi_s$ satisfies $1 \le \varphi_s(k)\le  k^s$, for each integer $k \ge 1$.

\smallskip

For $\varphi_s$ there is an identity analogous to that  of Lemma \ref{lemma:identity_double_sum_varphi_gcd} for $\varphi$:
\begin{lemma}\label{lemma:identity_double_sum_varphi_general_gcd}
For every real $t >1$, and for each integer $s \ge 1$
\begin{align}
\sum_{i,j=1}^\infty &\frac{\varphi_s(i)}{i^{s+t}}\frac{\varphi_s(j)}{j^{s+t}} \gcd(i,j)
=\zeta(2t-1) \zeta(t)^2\cdot
\\
\nonumber &\cdot \prod_{p}\Big(1-\frac{2}{p^{t+s}}-\frac{1}{p^{2t}}-\frac{2}{p^{2t+s-1}}+\frac{2}{p^{2t+s}}+\frac{1}{p^{2t+2s-1}}+\frac{2}{p^{3t+s-1}}-\frac{1}{p^{4t+2s-1}}\Big).
\end{align}
\end{lemma}
And a corresponding estimate:
\begin{corollary}\label{cor:partial_sums_identity_double_sum_varphi_general_gcd} For each integer $s \ge 1$
$$
\liminf_{N \to \infty} \frac{1}{\ln(N)^3} \sum_{\text{\rm lcm}(i,j) \le N}\frac{\varphi_s(i)}{i^{s+1}}\frac{\varphi_s(j)}{j^{s+1}} \gcd(i,j) \ge \Delta_s\, ,
$$
where $$
\Delta_s=\frac{1}{12} \prod_{p}\Big(1-\frac{4}{p^{s+1}}-\frac{1}{p^{2}}+\frac{4}{p^{s+2}}+\frac{1}{p^{2s+1}}-\frac{1}{p^{2s+3}}\Big)\, .
$$\end{corollary}

Of course, the constant $\Delta_1$ coincides with the constant $\Delta$ of Corollary \ref{cor:partial_sums_identity_double_sum_varphi_gcd}. The proof of Lemma \ref{lemma:identity_double_sum_varphi_general_gcd} proceeds along the same lines as that of Lemma \ref{lemma:identity_double_sum_varphi_gcd}, but using now the expression
$$
\frac{\varphi_s(k)}{k^{s}}=\prod_p \Big(1-\frac{h(\alpha_p(k))}{p^s}\Big)\, .
$$

\subsubsection{Asymptotic behavior of averages of $\varphi$ and of $\varphi_s$: Schur's constants}
The following lemma records the asymptotic behavior of certain averages of $\varphi$ and $\varphi_s$.
\begin{lemma}[Schur's constants]
%{\rm a)} For every integer $l \ge 1$,
%\begin{equation}
%S^{(1)}_l:=\lim_{n \to \infty} \frac{1}{n}\sum_{k=1}^n \Big(\frac{\varphi(k)}{k}\Big)^l=\prod_{p} \Big(1-\frac{1}{p}\Big[1-\Big(1-\frac{1}{p}\Big)^l\Big]\Big)
%\end{equation}
%
%{\rm b)}
For every integers $s,l \ge 1$,
\begin{equation}\label{eq:def of Schur constants}
S^{(s)}_l:=\lim_{n \to \infty} \frac{1}{n}\sum_{k=1}^n \Big(\frac{\varphi_s(k)}{k^s}\Big)^l=\prod_{p} \Big(1-\frac{1}{p}\Big[1-\Big(1-\frac{1}{p^s}\Big)^l\Big]\Big)\, .
\end{equation}
\end{lemma}

The case $s=1$ corresponds to the Euler $\varphi$ function. The particular case  $l=1$ reads $S^{(1)}_1=\prod_{p}\big(1-{1}/{p^2}\big)={1}/{\zeta(2)}$, and it is a direct consequence of the identity
$$
\frac{1}{n}\sum_{k=1}^n \frac{\varphi(k)}{k}=\frac{1}{n}\sum_{k=1}^n \sum_{j|k} \frac{\mu(j)}{j}=\sum_{j=1}^n \frac{\mu(j)}{j}\Big(\nfracj \frac{1}{n}\Big)\, .
$$
The case $l\ge 2$ is a result of Schur (see \cite{Kac}, page 58).

The results for~$\varphi_s$, with $s\ge 2$, may be obtained following the approach of~\cite{Kac}. Again, observe that $S^{(s)}_1={1}/{\zeta(s+1)}$. Notice, for later use,  that for any integers $s\ge 1$ and~$l\ge2$,
\begin{equation}
S^{(s)}_l > \big(S^{(s)}_1\big)^l \, ;
\end{equation}
strict inequality. Actually, in this paper, only the exponents $l=1,2$  are needed.

\subsection{Pillai's functions}
The arithmetic function of Pillai is defined for integer $k \ge 1$ as
$$
P(k)=\sum_{j=1}^k \gcd(j,k)\, .
$$
Observe that $P(k)$ may be written as $$
P(k)=\sum_{d|k}d \, \varphi\Big(\frac{k}{d}\Big)=\big(\varphi\ast I\big)(k)\,,
\quad\text{so that}\quad
\frac{P(k)}{k}=\sum_{d|k} \frac{\varphi(d)}{d}\, .
$$

Consider next, for each integer $s\ge 1$,  the arithmetic function $P_s$ given by the convolution $P_s=\varphi \ast I_s$; thus
$$
P_s(k)=\sum_{i=1}^k \gcd(i, k)^s\, .
$$
Observe that
\begin{equation}
\label{eq:s-Pillai as sum over divisors}\frac{P_s(k)}{k^s}=\sum_{d|k} \frac{\varphi(d)}{d^s}\, .
\end{equation}

The function $P_s$ may  be written also as $P_s=\varphi_s\ast I$ and interpreted alternatively as
$$
P_s(k)=\sum_{i_1, i_2, \ldots, i_s=1}^k \gcd(i_1, i_2, \ldots, i_s, k)\, .
$$

Notice also that
\begin{equation}
\label{eq:bound for s-Pillai}\frac{P_s(k)}{k^s}\le
\frac{P(k)}{k}\le \tau(k)\, , \quad \text{for each integer} \ k \ge 1\, .
\end{equation}
We refer to \cite{To2010} for further information on $P$ and $P_s$.

Although both $\varphi_s=\mu*I_s$ and $P_s=\varphi* I_s$ are well defined for real $s\ge 1$, we just consider the case $s$ integer.

\subsubsection{Asymptotic behavior of   averages  of $P$ and $P_s$}

We shall need the asymptotic behavior of averages, first and second moments, of $P$ and~$P_s$.

Since $P=\mu\ast I \ast I$, the Dirichlet series, with variable $z$,  of Pillai's function is given by:
$$
\sum_{k=1}^\infty \frac{P(k)}{k^z}=\frac{\zeta(z-1)^2}{\zeta(z)}\, , \quad \text{for $\Re(z)>2$}.
$$
Writing
$$
\sum_{k=1}^\infty \frac{P(k)}{k}\frac{1}{k^z}=\frac{\zeta(z)^2}{\zeta(z+1)}\, , \quad \text{for $\Re(z)>1$},
$$
we deduce directly, say from Delange's Theorem \ref{th:delange}, that
$$
\frac{1}{n}\sum_{k=1}^n \frac{P(k)}{k} \sim \frac{1}{\zeta(2)} \ln(n)\, , \quad \text{as} \ n \to \infty\, .
$$
For $s\ge 2$, we may write,  using $P_s=\mu\ast I\ast I_s$, that
$$
\sum_{k=1}^\infty \frac{P_s(k)}{k^s}\frac{1}{k^z}=\frac{\zeta(z)\zeta(z+s-1)}{\zeta(z+s)}\, , \quad \text{for $\Re(z)>1$},
$$
to deduce, as above, that
$$
\lim_{n \to \infty}\frac{1}{n}\sum_{k=1}^n \frac{P_s(k)}{k^s} =\frac{\zeta(s)}{\zeta(s+1)}\, .
$$
Thus,
\begin{lemma}\label{lemma:asymptotics_averages_order_1_pillais}
As $n \to \infty$,
\begin{align}
\frac{1}{n}\sum_{k=1}^n \frac{P(k)}{k} &\sim \frac{1}{\zeta(2)} \ln(n)\, ,
\\
\frac{1}{n}\sum_{k=1}^n \frac{P_s(k)}{k^s} &\rightarrow \frac{\zeta(s)}{\zeta(s+1)}\, , \quad \text{for} \ s \ge 2\, .
\end{align}
\end{lemma}

To obtain the asymptotic behavior of the averages of $P$ and $P_s$ for exponent $l=2$, we proceed as follows.
For $s \ge 1$, we may write
\begin{align*}
\frac{1}{n}\sum_{k=1}^n \Big(\frac{P_s(k)}{k^s}\Big)^2=
\frac{1}{n}\sum_{k=1}^n \Big(\sum_{j|k} \frac{\varphi(j)}{j^s}\Big)^2&=
\frac{1}{n}\sum_{k=1}^n \sum_{i,j|k} \frac{\varphi(i)}{i^s}\frac{\varphi(j)}{j^s}\\&=\sum_{1 \le i,j\le n}\frac{\varphi(i)}{i^s}\frac{\varphi(j)}{j^s}\bigg(\frac{1}{n}\Big\lfloor \frac{n}{\text{lcm}(i,j)}\Big\rfloor\bigg)\,
\end{align*}

For fixed integers $i,j\ge 1$, we have that
$$
\frac{1}{n}\Big\lfloor \frac{n}{\text{lcm}(i,j)}\Big\rfloor \le \frac{1}{\text{lcm}(i,j)}=\frac{\gcd(i,j)}{i\, j}\, ,
$$
and also that
$$
\lim_{n \to \infty} \frac{1}{n}\Big\lfloor \frac{n}{\text{lcm}(i,j)}\Big\rfloor=\frac{\gcd(i,j)}{i\, j}\, .
$$
We split the argument into the two cases $s\ge 2$ and $s=1$. For $s\ge 2$,  Lemma \ref{lemma:identity_double_sum_varphi_gcd} gives that
$$
\sum_{i,j=1}^\infty \frac{\varphi(i)}{i^{1+s}}\,\frac{\varphi(j)}{j^{1+s}} \,\gcd(i,j) <+\infty\, .
$$
and dominated convergence then gives that
$$
\lim_{n \to \infty}\frac{1}{n}\sum_{k=1}^n \Big(\frac{P_s(k)}{k^s}\Big)^2 =M(s)\, .
$$

For the case $s=1$, write
\begin{equation*}
\frac{1}{n}\sum_{k=1}^n \Big(\frac{P(k)}{k}\Big)^2=\sum_{\text{\rm lcm}(i,j)\le n}\frac{\varphi(i)}{i}\frac{\varphi(j)}{j}\, \Big(\frac{1}{n}\Big\lfloor \frac{n}{\text{lcm}(i,j)}\Big\rfloor\Big)\,.
\end{equation*}
(Notice the range of summation.)
Using that for any fixed integer $K \ge 2$, one has that $\lfloor x\rfloor \ge \big(1-1/K\big) x$, for any real $x\ge K$, we may bound
\begin{align*}
\frac{1}{n}\sum_{k=1}^n \Big(\frac{P(k)}{k}\Big)^2 &\ge\sum_{\text{\rm lcm}(i,j)\le \frac{n}{K}}\frac{\varphi(i)}{i}\frac{\varphi(j)}{j}\bigg(\frac{1}{n}\Big\lfloor \frac{n}{\text{lcm}(i,j)}\Big\rfloor\bigg)\\
&\ge \big(1-1/K\big)\sum_{\text{\rm lcm}(i,j)\le \frac{n}{K}}\frac{\varphi(i)}{i^2}\frac{\varphi(j)}{j^2}\gcd(i,j)\, .
\end{align*}
Using now Corollary \ref{cor:partial_sums_identity_double_sum_varphi_gcd} we may conclude that
\begin{equation*}
\liminf_{n \to \infty}\frac{1}{\ln(n)^{3}}\frac{1}{n}\sum_{k=1}^n \Big(\frac{P(k)}{k}\Big)^2 \ge \big(1-1/K\big) \Delta\, .
\end{equation*}
and, consequently,
\begin{equation*}
\liminf_{n \to \infty}\frac{1}{\ln(n)^{3}}\frac{1}{n}\sum_{k=1}^n \Big(\frac{P(k)}{k}\Big)^2 \ge  \Delta\, .
\end{equation*}

We record these results in the following:
\begin{lemma}\label{lemma:asymptotics_averages_order_2_pillais}
For $s=1$,
\begin{equation}\label{eq:pre-Toth}
\liminf_{n \to \infty}\frac{1}{\ln(n)^{3}}\frac{1}{n}\sum_{k=1}^n \Big(\frac{P(k)}{k}\Big)^2 \ge  \Delta\, ,
\end{equation}
while, for $s \ge 2$,
\begin{equation}
\lim_{n \to \infty}\frac{1}{n}\sum_{k=1}^n \Big(\frac{P_s(k)}{k^s}\Big)^2 =M(s)\, .
\end{equation}
\end{lemma}

\begin{remark}[A theorem of L. Toth]\label{remark:theorem_Toth} \upshape A result of L. Toth  (see \cite{To2010}, Theorem A) gives a precise version of \eqref{eq:pre-Toth}:
$$
\frac{1}{n}\sum_{k=1}^n \Big(\frac{P(k)}{k}\Big)^2 \sim \Delta_{\text{\rm Toth}} \ln(n)^3\, , \quad \text{as} \ n \to \infty\, ,
$$
where $\Delta_{\text{\rm Toth}}$ is the constant: $\Delta_{\text{\rm Toth}}=\frac{1}{\pi^2}\prod_{p} \big(1+\frac{1}{p^3}-\frac{4}{p(p+1)}\big)$.
Since for each $p$
$$
\Big(1+\frac{1}{p^3}-\frac{4}{p(p+1)}\Big) \Big(1-\frac{1}{p^2}\Big)=\Big(1-\frac{5}{p^2}+\frac{5}{p^3}-\frac{1}{p^5}\Big)\, ,
$$
actually, $
\Delta_{\text{\rm Toth}}=2 \Delta$.

\newpage

Observe that, from the discussion above and Toth's Theorem, one deduces that
$$
\sum_{\text{\rm lcm}(i,j)\le n} \frac{\varphi(i)}{i^2}\frac{\varphi(j)}{j^2} \gcd(i,j) \sim \Delta_{\text{\rm Toth}} \ln(n)^3\, , \quad \text{as} \ n \to \infty\, ,
$$
while, according to Corollary \ref{cor:partial_sums_identity_double_sum_varphi_gcd},
$$
\sum_{i\cdot j\le n} \frac{\varphi(i)}{i^2}\frac{\varphi(j)}{j^2} \gcd(i,j) \sim \Delta \ln(n)^3\, , \quad \text{as} \ n \to \infty\,  .
$$

\end{remark}

\section{Marginal probabilities and expectations}\label{section:marginal probs and expec} The  following lemma registers a couple of elementary but useful closed formulas for expectation of functions of $\gcd$. We shall call them \textit{Ces\`{a}ro's formulas} (see \cite{Ce1885} and \cite{Ce3}):
\begin{lemma}\label{lemma:cesaro_formulas}
Let $F$ be any arithmetic function.

\smallskip {\rm a)} {\upshape(Ces\`{a}ro's formula)} For any integers~$n,r\ge 1$,
\begin{equation}\label{eq:cesaro formula}
\E\big(F\big(\gcd\big(X^{(n)}_1, \ldots, X^{(n)}_r\big)\big)=\frac{1}{n^r}\sum_{j=1}^n \big(\mu\ast F\big)(j) \nfracj^r\, .
\end{equation}

{\rm b)} {\upshape(Ces\`{a}ro's marginal formula)} For any integers $1 \le k\le n$, and $r \ge 1$,
\begin{equation}\label{eq:cesaro_marginal_formula}
\E\big(F\big(\gcd\big(X^{(n)}_1, \ldots, X^{(n)}_{r},k\big)\big)=\frac{1}{n^{r}}\sum_{j |k} \big(\mu\ast F\big)(j) \nfracj^{r}\, .
\end{equation}
\end{lemma}
The expression \eqref{eq:cesaro formula} is valid also for $r=1$, with the conventional understanding that $\gcd(j)=j$, for any integer $j\ge 1$.
On the left hand side \eqref{eq:cesaro_marginal_formula} we have expectation marginal on $X^{(n)}_{r+1}=k$, while on the right the sum extends  only over divisors  of $k$.

\begin{proof}
They both follow from \eqref{eq:use of mu}.
For \eqref{eq:cesaro_marginal_formula}, the following observation is also needed: for each integer $k\ge 1$, define $F_k$ as the arithmetic function
$
F_k(j)=F(\gcd(j,k)).
$
Then for any integer $j \ge 1$,
$$
\big(\mu\ast F_k\big)(j)=\begin{cases}
\big(\mu\ast F\big)(j),& \ \ \text{if} \ j\mid k,\\
0,& \ \  \text{if} \ j \nmid k.
\end{cases}
$$
\end{proof}
For a fixed integer $k \ge 1$, Ces\`{a}ro's formula with $F=\delta_k$ reads
\begin{equation*}
\P\big(\gcd\big(X^{(n)}_1, \ldots, X^{(n)}_r\big)=k\big)=\frac{1}{n^r}\sum_{j \le {n}/{k}} \mu(j) \Big\lfloor \frac{n}{kj}\Big\rfloor^r\,,
\end{equation*}
from which one deduces the following asymptotic result, as $n \to \infty$, for the probability distribution of the  $\gcd$ of a random $r$-tuple:
\begin{equation*}
\lim_{n \to \infty}\P\big(\gcd\big(X^{(n)}_1, \ldots, X^{(n)}_r\big)=k\big)=\frac{1}{\zeta(r)}\frac{1}{k^r}\,,
\end{equation*} which, for the case $k=1$, reads
\begin{equation}\label{eq:probability_coprime_r_a_r}
\lim_{n \to \infty}\P\big(\big(X^{(n)}_1, \ldots, X^{(n)}_r\big) \ \textit{\rm coprime}\big)=\frac{1}{\zeta(r)}\,,
\end{equation}
The case $r=2$ is Dirichlet's Theorem. For $r\ge 3$, see
Ces\`{a}ro \cite{Ce3} (page 293), D.\,N. Lehmer~\cite{Lehmer} (Chapter~V), and also \cite{Christopher}, \cite{HeSt} or~\cite{Nymann}.

\newpage

If we set $F=I$ in Ces\`{a}ro's formula \eqref{eq:cesaro formula} we obtain,  since $\mu \ast I=\varphi$, that
\begin{equation*}
\E\big(\gcd\big(X^{(n)}_1, \ldots, X^{(n)}_r\big)\big)=\frac{1}{n^r}\sum_{j=1}^n \varphi(j) \Big\lfloor \frac{n}{j}\Big\rfloor^r\, ,
\end{equation*}
from which the following asymptotic results for the expectation of the $\gcd$ of a random  $r$-tuple are deduced: for $r \ge 3$,
\begin{equation}\label{eq:expectation_gcd_r_a_r_bigger_3}
\lim_{n \to \infty}\E\big(\gcd\big(X^{(n)}_1, \ldots, X^{(n)}_r\big)\big)=\sum_{j=1}^{\infty}\frac{\varphi(j)}{j^r}=\frac{\zeta(r-1)}{\zeta(r)}\, ,
\end{equation}
while for $r=2$,
\begin{equation}\label{eq:expectation_gcd_r_a_r_equal_2}
\E\big(\gcd\big(X^{(n)}_1, X^{(n)}_2\big)\big)\sim\frac{1}{\zeta(2)}\ln(n)\, , \quad \text{as} \ n \to \infty\, .
\end{equation}

For the second moments of $\gcd$, which we shall need later on, we have (see for instance Theorem~A' in ~\cite{FF}, and the references therein):
\begin{align}
\text{for $r \ge 4$,}&\qquad\lim_{n \to \infty} \E\big(\gcd\big(X^{(n)}_1, X^{(n)}_2, \ldots,X^{(n)}_r\big)^2 \big)=\frac{\zeta(r-2)}{\zeta(r)}\, ;
\\
\text{for $r =3$,}&\qquad\E\big(\gcd\big(X^{(n)}_1, X^{(n)}_2, X^{(n)}_3\big)^2 \big)\sim \frac{1}{\zeta(3)} \ln(n)\quad\text{as $n\to\infty$};
\\
\text{for $r =2$,}&\qquad \label{eq:second_moment_gcd_pairs}
\E\big(\gcd\big(X^{(n)}_1, X^{(n)}_2\big)^2 \big)\sim \Big[\frac{1}{3}\Big(\frac{2\zeta(2)}{\zeta(3)}-1\Big)\Big] n
\quad\text{as $n\to\infty$}.
\end{align}

Two particularly relevant cases of the marginal formula \eqref{eq:cesaro_marginal_formula} are obtained by setting $F=\delta_1$ and~$F=I$. They will appear quite often along this paper; specific notations are in order.

\medskip

[\textit{Marginal probability}] With $F=\delta_1$, we obtain, for $1 \le k \le n$
\begin{equation}\label{eq:formula_marginal_probability}
\unrk:=\P\big(\gcd\big(X^{(n)}_1, \ldots, X^{(n)}_{r},k\big)=1\big)=\frac{1}{n^r}\sum_{j|k} \mu(j) \nfracj^r\, .
\end{equation}

For $r=1$, $U_1^{(n)}(k)$ is the proportion of numbers in $\{1,\dots,n\}$ which are coprime with $k$, the familiar Legendre function.
\medskip

[\textit{Marginal expectation}] With $F=I$, we obtain, for $1 \le k \le n$,
\begin{equation}\label{eq:formula_marginal_expectation}
\wnrk:=\E\big(\gcd\big(X^{(n)}_1, \ldots, X^{(n)}_{r},k\big)\big)=\frac{1}{n^r}\sum_{j|k} \varphi(j) \nfracj^r\, .
\end{equation}

\subsection{Estimates and asymptotic behavior}\label{subsection:estimates and asymptotic behavior}

In this section we record some estimates for the marginal probabilities $\unrk$ and expectations $\wnrk$ that we shall need later on. Both estimates come about from comparing their respective expressions \eqref{eq:formula_marginal_probability} and \eqref{eq:formula_marginal_expectation} with the analogous expressions that you get by removing the floor $\lfloor  \ \rfloor$, and which are quite more manageable as they do not contain $n$.

\smallskip
In this section, expectations and variances with respect to the uniform probability in $\{1, \ldots, n\}$ will be denoted by $\E_n$ and $\V_n$, respectively. Thus for a function (random variable) $f$ defined on $\{1, \ldots, n\}$, we have, for instance, $\E_n(f)=\frac{1}{n}\sum_{j=1}^n f(j)$.

\begin{lemma}[Estimates for marginal probabilities and expectations]\label{lemma:estimate_marginal}
For any integers $n, r \ge 1$ and any integer $k$, with $1 \le k\le n$, we have that
\begin{align}
\Big|\unrk- \sum_{j | k}\frac{\mu(j)}{j^r}\Big|&=\Big|\unrk-\frac{\varphi_r(k)}{k^r}\Big|\le r\,\frac{\tau(k)}{n}\, ,\label{eq:estimate for U_r^n(k)}
\\ \label{eq:estimate for W_r^n(k)}
0\le \sum_{j | k}\frac{\varphi(j)}{j^r}- \wnrk&=\frac{P_r(k)}{k^r}-\wnrk \le
\begin{cases}
{k}/{n}\, , & \text{if} \quad r=1\, ,
\\[6pt]
r \,{\tau(k)}/{n}\, , & \text{if} \quad r\ge 2\, .
\end{cases}
\end{align}
%and, besides,
%$$
%\wnrk \le\sum_{j | k}\frac{\varphi(j)}{j^r}=\frac{P_r(k)}{k^r}\, .$$
\end{lemma}
%In particular, for $r=1$,
%\begin{equation*}
%\Big|U^{(n)}_1(k)-\frac{\varphi(k)}{k}\Big| \le \frac{\tau(k)}{n}\quad \text{and} \quad
%0\le \frac{P(k)}{k}-W^{(n)}_1(k)\le \frac{k}{n}\, .
%\end{equation*}

These, of course, are standard bounds. See, for instance, D.\,H. Lehmer (Lemma~4 in~\cite{Lehmer-son}) or Toth (equation~(7) in~\cite{To2004}) for the case $r=1$ of \eqref{eq:estimate for U_r^n(k)}.

\begin{proof} We shall use that $x^r-\lfloor x\rfloor^r \le r x^{r-1}$, for any $x >0$.

\smallskip
a) We may bound
\begin{align*}
\Big|\sum_{j | k}\frac{\mu(j)}{j^r}-\unrk\Big|=\frac{1}{n^r}\Big|\sum_{j|k} \mu(j) \Big(\Big(\frac{n}{j}\Big)^r-\nfracj^r\Big)\Big|\le\frac{r}{n}\sum_{j|k} \frac{1}{j^{r-1}}\le \frac{r}{n}\sum_{j|k} 1=r\frac{\tau(k)}{n}\, .
\end{align*}

b)  The fact that $\wnrk\le P_r(k)/k^r$ is immediate (see  \eqref{eq:s-Pillai as sum over divisors}). Finally,
\begin{align*}
\sum_{j | k}\frac{\varphi(j)}{j^r}-\wnrk&=\frac{1}{n^r}\sum_{j|k} \varphi(j) \Big(\Big(\frac{n}{j}\Big)^r-\nfracj^r\Big)\le \frac{r}{n}\sum_{j|k} \frac{\varphi(j)}{j^{r-1}}\,.
\end{align*}
For $r=1$, we use that $\sum_{j|k} \varphi(j)=k$, while, for $r\ge 2$, we use that  $\sum_{j|k} {\varphi(j)}/{j^{r-1}}\le \sum_{j|k} {\varphi(j)}/{j}\le \tau(k)$.
\end{proof}

\subsubsection{Asymptotic behavior of means}

The  average values of the marginal probabilities and expectations are, simply,
\begin{align*}
\mu^{(n)}_r&:=\E_n(\unr)=\frac{1}{n}\sum_{k=1}^n \unrk =\P\big(\gcd(X^{(n)}_1, \ldots, X^{(n)}_r, X^{(n)}_{r+1})=1\big)\,,
\\
\nu^{(n)}_r&:=\E_n(\wnr)=\frac{1}{n}\sum_{k=1}^n \wnrk =\E\big(\gcd(X^{(n)}_1, \ldots, X^{(n)}_r, X^{(n)}_{r+1})\big)\,.
\end{align*}
From equations \eqref{eq:probability_coprime_r_a_r},  \eqref{eq:expectation_gcd_r_a_r_bigger_3} and \eqref{eq:expectation_gcd_r_a_r_equal_2}, we have:
\begin{lemma}\label{lemma:limit_mean_marginals}
For each integer $r\ge 1$,
\begin{equation*}
\lim_{n\to \infty} \mu^{(n)}_r=\frac{1}{\zeta(r+1)}\,.   %=S^{(1)}_r\, .
\end{equation*}
For each integer $r\ge 2$,
\begin{equation*}
\lim_{n\to \infty} \nu^{(n)}_r=\frac{\zeta(r)}{\zeta(r+1)}\, ,
\end{equation*}
while, for $r=1$,
\begin{equation*}
\nu^{(n)}_1 \sim \frac{1}{\zeta(2)} \ln(n)
\end{equation*}
\end{lemma}

\subsubsection{Asymptotic behavior of variances} We denote by $c^{(n)}_r$ the variance of the marginal probability $\unr$:
\begin{equation*}
c^{(n)}_r:=\V_n(\unr)=\E_n({\unr}^2)-\E_n(\unr)^2=\frac{1}{n}\sum_{k=1}^n {\unrk}^2-\Big(\frac{1}{n}\sum_{k=1}^n \unrk\Big)^2\, .
\end{equation*}
Observe that we may interpret $c^{(n)}_r$ as the covariance
\begin{equation}\label{eq:cnr_as_covariance}
c^{(n)}_r=\text{\rm cov}\big(\unogrande_{\gcd(X^{(n)}_1,X^{(n)}_2, \ldots, X^{(n)}_r, X^{(n)}_{r+1})=1}\, , \,  \unogrande_{\gcd(X^{(n)}_1,X^{(n)}_{r+2},  X^{(n)}_{r+3}, \ldots, X^{(n)}_{2r+1})=1}\big)\, ,
\end{equation}
where each of the two $\gcd$'s involves $r+1$ among the $X^{(n)}_j$'s variables, sharing exactly one of them, $X^{(n)}_1$. This interpretation follows by conditioning on the value of the common variable~$X^{(n)}_1$.

\smallskip

Appealing to the estimate of Lemma \ref{lemma:estimate_marginal}, we may compare second moments as follows
\begin{align*}
\Big|\frac{1}{n}\sum_{k=1}^n {\unrk}^2-\frac{1}{n}\sum_{k=1}^n \Big(\frac{\varphi_r(k)}{k^r}\Big)^2 \Big| &\le
\frac{1}{n}\sum_{k=1}^n \Big|\unrk-\frac{\varphi_r(k)}{k^r}\Big|\Big|\unrk+\frac{\varphi_r(k)}{k^r}\Big|\\&\le
\frac{1}{n}\sum_{k=1}^n \Big(r \frac{\tau(k)}{n}\Big) \,2= \frac{2r}{n^{2}}\sum_{k=1}^n \tau(k)\le \frac{2r}{n} (1+\ln(n))\, ,
\end{align*}
where, besides, we have used that $\unrk \le 1$, that $\varphi_r(k)\le k^r$ and also that $\sum_{k=1}^n\tau(k) \le n (1+\ln(n))$ (see for instance \cite{HW}, Theorem 320).

\smallskip

From this, and recalling the definition \eqref{eq:def of Schur constants} of Schur's constant $S^{(r)}_2$, we deduce that
$$
\lim_{n \to \infty}\frac{1}{n}\sum_{k=1}^n {\unrk}^2=S^{(r)}_2\, ,
$$
and, consequently, in conjunction with Lemma \ref{lemma:limit_mean_marginals}, and since $S_r^{(1)}=1/\zeta(r+1)$,
\begin{lemma}\label{lemma:asymptotics_c_n}For any integer $r \ge 1$,
\begin{equation*}
\lim_{n \to \infty}c^{(n)}_r=\lim_{n \to \infty}\V_n(\unr)=S^{(r)}_2-(S^{(r)}_1)^2 \,.
\end{equation*}
\end{lemma}

We point out for later use that $\lim_{n \to \infty}c^{(n)}_r>0$.

\smallskip

The analysis of the variance of the marginal expectation is a bit more involved. We introduce the notation $d^{(n)}_r$ for the variance of the marginal expectation $\wnr$:
\begin{equation*}
d^{(n)}_r:=\V_n(\wnr)=\E_n({\wnr}^2)-\E_n(\wnr)^2=\frac{1}{n}\sum_{k=1}^n {\wnrk}^2-\Big(\frac{1}{n}\sum_{k=1}^n \wnrk\Big)^2\, .
\end{equation*}
Observe that we may interpret $d^{(n)}_r$ as the covariance
\begin{equation}\label{eq:dnr_as_covariance}
d^{(n)}_r=\text{\rm cov}\big({\gcd\big(X^{(n)}_1,X^{(n)}_2, \ldots, X^{(n)}_r, X^{(n)}_{r+1}\big)}\, , \,  {\gcd\big(X^{(n)}_1,X^{(n)}_{r+2},  X^{(n)}_{r+3}, \ldots, X^{(n)}_{2r+1}\big)}\big)\, ;
\end{equation}
where, again, each of the two $\gcd$'s involves $r+1$ among the $X^{(n)}_j$'s variables, sharing exactly one of them, $X^{(n)}_1$.
%This alternative interpretation as a covariance follows again by conditioning on the value of the common variable $X^{(n)}_1$.

\newpage
We compare second moments as follows:
\begin{align*}
\Big|\frac{1}{n}\sum_{k=1}^n {\wnrk}^2-\frac{1}{n}\sum_{k=1}^n \Big(\frac{P_r(k)}{k^r}\Big)^2 \Big| &\le
\frac{1}{n}\sum_{k=1}^n \Big|\wnrk-\frac{P_r(k)}{k^r}\Big|\Big|\wnrk+\frac{P_r(k)}{k^r}\Big|
\\&\le
\frac{1}{n}\sum_{k=1}^n \Big|\wnrk-\frac{P_r(k)}{k^r}\Big|\, (2\, \tau(k))\,,
\end{align*}
where we have used that $\wnrk \le \frac{P_r(k)}{k^r} \le \tau(k)$ (see \eqref{eq:bound for s-Pillai} and \eqref{eq:estimate for W_r^n(k)}).

\smallskip
Now, we appeal to Lemma \ref{lemma:estimate_marginal}. For $r=1$, we obtain that
\begin{equation}\label{eq:estimate_w^2_r=1}
\Big|\frac{1}{n}\sum_{k=1}^n {\wnrk}^2-\frac{1}{n}\sum_{k=1}^n \Big(\frac{P_r(k)}{k^r}\Big)^2 \Big|\le \frac{2}{n^2}\sum_{k=1}^n k\, \tau(k)
\le \frac{2}{n}\sum_{k=1}^n \tau(k)\le 2 (1+\ln(n))\, ,
\end{equation}
where we have used that $k\le n$ and, once again, that $\sum_{k=1}^n \tau(k)\le n (1+\ln(n))$. Notice that for $r=1$ the bound obtained does not converge to 0.

For $r \ge 2$, we have that
\begin{equation}\label{eq:estimate_w^2_r=2}
\Big|\frac{1}{n}\sum_{k=1}^n {\wnrk}^2-\frac{1}{n}\sum_{k=1}^n \Big(\frac{P_r(k)}{k^r}\Big)^2 \Big|\le \frac{2}{n}\sum_{k=1}^n r\frac{\tau(k)}{n}\tau(k)
=\frac{2r}{n^2}\sum_{k=1}^n \tau(k)^2\, .
\end{equation}
This bound does converge to 0, as $n \to \infty$; this maybe be seen by recalling that $\tau(k)=O_{\delta}(k^\delta)$, for any $\delta >0$ (see \cite{HW}, Theorem 315), or more precisely, by appealing to Ramanujan's asymptotic result that $\sum_{k=1}^n \tau(k)^2 \sim \frac{1}{2\zeta(2)}n (\ln(n))^3$, as $n\to \infty$  (see \cite{HW}, second note on Chapter XVIII and the references therein). Incidentally, the bound $\sum_{k=1}^n k\tau(k)$ of \eqref{eq:estimate_w^2_r=1}  behaves asymptotically as $\frac{1}{2} n^2 \ln(n)$, as $n \to \infty$, since $\sum_{k=1}^\infty \frac{k\,\tau(k)}{k^z}=\zeta(z-1)^2$ for $\Re(z)>2$. In any case, in  what follows we just need that the bound in \eqref{eq:estimate_w^2_r=1} is $o(\ln(n))^3$ and that the bound in \eqref{eq:estimate_w^2_r=2} is $o(1)$.

\smallskip

We keep splitting the discussion into the case $r=1$ and the case $r \ge 2$. We start with the latter.

If $r \ge 2$, the bound in equation \eqref{eq:estimate_w^2_r=2} converges to 0, as $n\to \infty$. Moreover, in this case, Lemma \ref{lemma:asymptotics_averages_order_2_pillais} gives
$$
\lim_{n \to \infty}\frac{1}{n}\sum_{k=1}^n \Big(\frac{P_r(k)}{k^r}\Big)^2=M(r)\,.
$$
We conclude that also
$$
\lim_{n \to \infty}\frac{1}{n}\sum_{k=1}^n {\wnrk}^2=M(r)\, ,
$$
and, therefore,  that
$$
\lim_{n \to \infty}\V_n(W_r^{(n)})=M(r)-\Big(\frac{\zeta(r)}{\zeta(r+1)}\Big)^2\, .
$$
Since
$$
\frac{\zeta(r)}{\zeta(r+1)}=\sum_{j=1}^{\infty}\frac{\varphi(j)}{j^{r+1}}\, ,
$$
we may, finally,  write this limiting variance in the following appealing form:
$$
\lim_{n \to \infty}\V_n(\wnr)=\sum_{1 \le i,j < \infty} \frac{\varphi(i)}{i^{r+1}}\frac{\varphi(j)}{j^{r+1}}\, \big(\gcd(i,j)-1\big)\, .
$$

Now we turn to the case $r=1$. Notice that the bound in equation  \eqref{eq:estimate_w^2_r=1}  is of the order~$\ln(n)$, while, by Toth's Theorem, see Remark \ref{remark:theorem_Toth},   the average $\frac{1}{n}\sum_{k=1}^n ({P(k)}/{k})^2$ is of order $\ln (n)^3$. Therefore,
$$
\frac{1}{n}\sum_{k=1}^n {\wnrk}^2 \sim \Delta_{\text{\rm Toth}} \ln(n)^3\, .
$$
Finally, since
$$
\Big(\frac{1}{n}\sum_{k=1}^n {\wnrk}\Big)^2=\big(\nu_1^{(r)}\big)^2 \sim \frac{1}{\zeta(2)^2} \ln(n)^2\, ,
$$
we conclude that
$$
\V_n(W^{(n)}_1)\sim \Delta_{\text{\rm Toth}} \ln(n)^3\,.
$$

We have proved:
\begin{lemma}\label{lemma:asymptotics_d_n}
For any integer $r \ge 2$,
$$
\lim_{n \to \infty} d^{(n)}_r=
\lim_{n \to \infty}\V_n(\wnr)=\sum_{1 \le i,j < \infty} \frac{\varphi(i)}{i^{r+1}}\frac{\varphi(j)}{j^{r+1}}\,\big(\gcd(i,j)-1\big)\, ,
$$
while, for $r=1$,
$$
d^{(n)}_1=\V_n(W^{(n)}_1) \sim \Delta_{\text{\rm Toth}} \ln(n)^3\, , \quad \text{as} \ n \to \infty\, .
$$
\end{lemma}

\section{Statistics of gcd of pairs}  \label{section:statistics_gcd_pairs}
Equipped with the estimates that we have gathered in the last two sections, in particular, Lemmas  \ref{lemma:limit_mean_marginals}, \ref{lemma:asymptotics_c_n} and \ref{lemma:asymptotics_d_n}, we are now ready to tackle the statistics of $\gcd$ of pairs of large sample of integers.
We shall focus in the limiting behavior, as the sample size $m$ tends to infinity,  of the distribution of the following three  basic statistics:
$$
\mathcal{C}^{(n)}_m=\sum_{1\le i<j\le m} \text{\large\bf 1}_{\gcd(X^{(n)}_i,X^{(n)}_j)=1}\,,
$$
that counts the number of coprime pairs,
$$
\mathcal{Z}^{(n)}_m=\sum_{1\le i<j\le m} \gcd(X^{(n)}_i,X^{(n)}_j)\, ,$$
which sums the $\gcd$ of pairs of the sample, and
$$
\mathcal{M}^{(n)}_m=\max_{1\le i<j\le m} \{\gcd(X^{(n)}_i,X^{(n)}_j)\}\, ,
$$
which gives the maximum $\gcd$ of the pairs of the sample.

We should remark that the asymptotic results of Section \ref{subsection:estimates and asymptotic behavior} that pertain to this section on pairs are those with $r=1$ (and not $r=2$).

\subsection{Asymptotic distribution of the number of coprime couples}\label{subsection:distribucion indicadores}
We start with the counter
$$
\mathcal{C}^{(n)}_m=\sum_{1\le i<j\le m} \text{\large\bf 1}_{\gcd(X^{(n)}_i,X^{(n)}_j)=1},
$$
a sum of $\binom{m}{2}$ random variables, identically distributed  but not independent.

\newpage
The  expectation of $\mathcal{C}^{(n)}_m$ is given by
\begin{equation}
\E(\mathcal{C}^{(n)}_m)=\binom{m}{2} \,\E(\text{\large\bf 1}_{\gcd(X^{(n)}_1,X^{(n)}_2)=1})=\binom{m}{2} \mu^{(n)}_1\, .
\end{equation}
Recall (Lemma \ref{lemma:limit_mean_marginals}) that, as $n\to \infty$,
$$
\E\big(\text{\large\bf 1}_{\gcd(X^{(n)}_1,X^{(n)}_2)=1}\big)= \mu^{(n)}_1 \underset{n \to \infty}{\longrightarrow} \frac{1}{\zeta(2)}\, .
$$
Notice also that
$$
\V\big(\text{\large\bf 1}_{\gcd(X^{(n)}_1,X^{(n)}_2)=1}\big)=\mu^{(n)}_1(1-\mu^{(n)}_1).
$$

For the variance of $\mathcal{C}^{(n)}_m$ we have:
\begin{lemma}
\label{lema:varianza de indicadores} The variance of the variable $\mathcal{C}^{(n)}_m$ is given by
\begin{equation}\label{eq:varianza de indicadores}
\V(\mathcal{C}^{(n)}_m)=\binom{m}{2} \, \mu^{(n)}_1\, (1-\mu^{(n)}_1)+m(m-1)(m-2) \, c^{(n)}_1\, .
\end{equation}
\end{lemma}
Recall that
$$
\lim_{n \to \infty}c^{(n)}_1=\lim_{n \to \infty}\V_n(U^{(n)}_1)=S^{(1)}_2-\big(S^{(1)}_1\big)^2 >0\,
$$
(see Lemma \ref{lemma:asymptotics_c_n}). Therefore, since $c^{(n)}_1>0$ for each $n \ge 2$,  we have that
\begin{equation}
\label{eq:def of C1}
C_1:=\inf_{n \ge 2} c^{(n)}_1>0.
\end{equation}
\begin{proof}
The variable  $\mathcal{C}^{(n)}_m$ is a sum of  $\binom{m}{2}$ terms, so there will appear $\binom{m}{2}^2$ terms in the expansion of its variance in  terms of covariances of pairs of summands:
\begin{itemize}
\item $\binom{m}{2}$ individual variances $\V(\text{\large\bf 1}_{\gcd(X^{(n)}_1,X^{(n)}_2)=1})=\mu^{(n)}_1\, (1-\mu^{(n)}_1)$;
\item $m(m-1)(m-2)$ covariances of the type
$$
\text{cov}\big(\text{\large\bf 1}_{\gcd(X^{(n)}_1,X^{(n)}_2)=1}, \text{\large\bf 1}_{\gcd(X^{(n)}_2,X^{(n)}_3)=1}\big)=\V_n(U^{(n)}_1)=c^{(n)}_1\, ,
$$
(with \textit{exactly} one $X^{(n)}_j$ in common);
\item plus $\binom{m}{2}\binom{m-2}{2}$ covariances of the type
$$
\text{cov}\big(\text{\large\bf 1}_{\gcd(X^{(n)}_1,X^{(n)}_2)=1}, \text{\large\bf 1}_{\gcd(X^{(n)}_3,X^{(n)}_4)=1}\big)
$$
(with no $X^{(n)}_j$ in common). All these covariances are 0, because of the independence of the $X^{(n)}_j$'s.
\end{itemize}
Equation \eqref{eq:varianza de indicadores} follows. \end{proof}

\smallskip

Consider now the collection of $\binom{m}{2}$ variables $\text{\large\bf 1}_{\gcd(X^{(n)}_i, X^{(n)}_j)=1}$, with $i<j$, as the vertices of a graph $\Gamma^{(n)}_m$; there is an edge joining a pair of vertices $\text{\large\bf 1}_{\gcd(X^{(n)}_i, X^{(n)}_j)=1}$ and $\text{\large\bf 1}_{\gcd(X^{(n)}_k, X^{(n)}_l)=1}$  if the sets of indexes
$\{i,j\}$ and $\{k,l\}$ have exactly one index in common. Thus $\Gamma^{(n)}_m$ is the dependency graph of the variables $\{\text{\large\bf 1}_{\gcd(X^{(n)}_i, X^{(n)}_j)=1}\}_{1 \le i <j \le n}$.

We will now apply an asymptotic normality result of S. Janson \cite{Ja1988}, see also P. Baldi and Y. Rinnot (\cite{BR1}, particularly Proposition 5),  concerning sums of (locally) dependent variables.

\begin{theorem}[Janson, Theorem 2 in \cite{Ja1988}] \label{th:Janson} Suppose that,  for each integer $t\ge 1$, we have  a family $\{Y_{i_1},\dots, Y_{i_{N_t}}\}$  of $N_t$ bounded random variables, with almost sure common  bound $|Y_{i_j}|\le A_t$.  Let~$M_t$ be the maximal degree of the dependency graph  $\Gamma_t$ of the family $\{Y_{i_1},\dots, Y_{i_{N_t}}\}$. Denote
$$
S_t=\sum_{j=1}^{N_t} Y_{i_j}
$$
and let $\sigma_t^2=\V(S_t)$. If there exists an integer $h\ge 3$ such that
\begin{equation}
\label{eq:condicion de Janson}
\Big(\frac{N_t}{M_t}\Big)^{1/h} \ \frac{M_t\, A_t}{\sigma_t}\to 0 \quad\text{as $t\to\infty$,}
\end{equation}
then
$$
\frac{S_t-\E(S_t)}{\sigma_t}\overset{\text{d}}{\longrightarrow}\mathcal{N}\, , \quad\text{as $t\to\infty$.}
$$
\end{theorem}

For the families $\big\{\text{\large\bf 1}_{\gcd(X^{(n)}_i, X^{(n)}_j)=1}\big\}_{1 \le i <j \le n}$ with dependency graphs $\Gamma^{(n)}_m$, the corresponding parameters of Janson's Theorem are: number of variables $N=\binom{m}{2}$, maximal degree: $M=2(m-2)$, uniform bound on the variables $A=1$,  and $$
\sigma^2=\binom{m}{2} \, \mu^{(n)}_1\, (1-\mu^{(n)}_1)+m(m-1)(m-2) \, c^{(n)}_1\,.
$$
Only in this last parameter the size $n$ of the sample space intervenes, but actually, we may bound
$$
\sigma^2 \ge m(m-1)(m-2) \,C_1\, ,
$$
where $C_1$ is given in \eqref{eq:def of C1}, as long as $n \ge 2$. Now,
$$
\Big(\frac{N}{M}\Big)^{1/h} \ \frac{M\, A}{\sigma}\le \Big(\frac{\binom{m}{2}}{2(m-2)}\Big)^{1/h} \frac{2(m-2)}{\sqrt{m(m-1)(m-2) C_1}}
\underset{m \to \infty}{\longrightarrow}  0\, ,
$$
as long as the \textit{integer} $h \ge 3$. Summarizing, we have proved:
\begin{theorem}\label{th:TLC para indicadores}
The counter of coprime pairs $\mathcal{C}^{(n)}_m$ is asymptotically normal:
$$
\frac{\mathcal{C}^{(n)}_m-\E(\mathcal{C}^{(n)}_m)}{\sqrt{\V(\mathcal{C}^{(n)}_m)}} \overset{\text{d}}{\longrightarrow} \mathcal{N}\, , \quad \text{as} \quad m \to \infty\, ,
$$
for any fixed $n \ge 2$. More generally,
$$
\frac{\mathcal{C}^{(n_m)}_m-\E(\mathcal{C}^{(n_m)}_m)}{\sqrt{\V(\mathcal{C}^{(n_m)}_m)}} \overset{\text{d}}{\longrightarrow} \mathcal{N}\, , \quad \text{as} \quad m \to \infty\, ,
$$
for any sequence $n_m$ as long as $n_m \ge 2$, for each $m$.
\end{theorem}

It is perhaps more natural to consider $\widetilde{\mathcal{C}}^{(n)}_m={\binom{m}{2}}^{-1} \mathcal{C}^{(n)}_m$, the average number of coprime pairs in the sample of size $m$. For fixed $n\ge 2$, we have as, $m\to \infty$,
$$
\dfrac{\widetilde{\mathcal{C}}^{(n)}_m-\frac{1}{\zeta(2)}}{{2 \sqrt{c^{(n)}_1/m}}} \overset{\text{d}}{\longrightarrow} \mathcal{N}\, .
$$

So, for large $m$ (the sample size) and $n$  (the size of the sample space),
$$
\widetilde{\mathcal{C}}^{(n)}_m \approx \mathcal{N}\bigg(\frac{1}{\zeta(2)},\frac{2}{\sqrt{m}}\sqrt{S^{(1)}_2-\frac{1}{\zeta(2)^2}}\bigg)\, .
$$

\begin{remark}
\upshape Notice that, for $n$ fixed, the variance of ${\mathcal{C}}^{(n)}_m$ is of the order ${\binom{m}{2}}^{3/2}$. This is to be compared with the variance of a sum of $\binom{m}{2}$ identically distributed and pairwise independent Bernoulli variables, which is of the order $\binom{m}{2}$, and with the variance of a sum of $\binom{m}{2}$ identically distributed Bernoulli variables with constant positive correlation among them, which is of the order $\binom{m}{2}^2$.
\end{remark}

\begin{remark}
\upshape As we have mentioned in the introduction, Theorem~\ref{th:TLC para indicadores} could be derived in  the $n$ constant case  from  the classical results of W. Hoeffding on normal approximation of $U$ statistics,  \cite{Hoeffding}, see also~\cite{Serfling}. The approach through dependency graphs appears to be more flexible, particularly when $n$ is allowed to vary. In any case, asymptotic normality of standard $U$ statistics could be derived from the dependency graph approach, see Application~C in \cite {BR1}.
\end{remark}

\begin{remark}\upshape
There are good estimates for the \textit{rate of converge to normality} for sums of locally dependent variables, for instance, \cite{BR2}, which could be applied to the variables $\mathcal{C}^{(n)}_m$.
\end{remark}

\subsection{Sums of greatest common divisor of pairs}\label{subseccion:gcd of pairs}
We now deal with the random variable
$$
\mathcal{Z}^{(n)}_m=\sum_{1\le i<j\le m} \gcd(X^{(n)}_i,X^{(n)}_j)\, .
$$

The mean of $\mathcal{Z}^{(n)}_m$  is given by
$$
\E(\mathcal{Z}^{(n)}_m)=\binom{m}{2}\, \nu^n_1,
$$
where
$$
\nu^{(n)}_1=\E(\gcd(X^{n}_1,X^{(n)}_2)=\E(W^{(n)}_1)\sim \frac{1}{\zeta(2)}\ln(n)\, , \quad \text{as} \ n \to \infty\, .
$$

Recall, see \eqref{eq:second_moment_gcd_pairs}, that
$$
\E\big(\gcd\big(X^{(n)}_1, X^{(n)}_2\big)^2 \big)\sim \Big[\frac{1}{3}\Big(\frac{2\zeta(2)}{\zeta(3)}-1\Big)\Big] n\, ,
$$
and consequently, that, also,
$$
\V\big(\gcd\big(X^{(n)}_1, X^{(n)}_2\big)\big)\sim \Big[\frac{1}{3}\Big(\frac{2\zeta(2)}{\zeta(3)}-1\Big)\Big] n\, .
$$

For the variance of $\mathcal{Z}^{(n)}_m$,  we have, with the same argument as in Lemma \ref{lema:varianza de indicadores}:
\begin{lemma}
\label{lema:varianza de gcd} The variance of the variable $\mathcal{Z}^{(n)}_m$ is given by
\begin{equation}\label{eq:varianza de gcd}
\V(\mathcal{Z}^{(n)}_m)=\binom{m}{2} \, \V\big(\gcd\big(X^{(n)}_1, X^{(n)}_2\big)\big)+m(m-1)(m-2) \, d^{(n)}_1\, .
\end{equation}
\end{lemma}
This follows since
$$
\text{\rm cov}(\gcd(X^{(n)}_1, X^{(n)}_2), \gcd(X^{(n)}_1, X^{(n)}_3)=d^{(n)}_1
$$
(see \eqref{eq:dnr_as_covariance}).
Recall from Lemma \ref{lemma:asymptotics_d_n} that
$$
d^{(n)}_1 \sim \Delta_{\text{\rm Toth}} \ln(n)^3\, , \quad \text{as} \ n \to \infty\, .
$$

With all this, we can now prove:
\begin{theorem}\label{th:TLC para gcd}
The sum of $\gcd$ of pairs, $\mathcal{Z}^{(n)}_m$, satisfies
$$
\frac{\mathcal{Z}^{(n)}_m-\E(\mathcal{Z}^{(n)}_m)}{\sqrt{\V(\mathcal{Z}^{(n)}_m)}} \overset{\text{d}}{\longrightarrow} \mathcal{N}\, , \quad \text{as} \ m \to \infty
$$ for any fixed $n \ge 2$. More generally,
$$
\frac{\mathcal{Z}^{(n_m)}_m-\E(\mathcal{Z}^{(n_m)}_m)}{\sqrt{\V(\mathcal{Z}^{(n_m)}_m)}} \overset{\text{d}}{\longrightarrow} \mathcal{N}\, , \quad \text{as} \ m \to \infty
$$
for any sequence $n_m$ as long as $n_m\ge 2$ for each $m \ge 1$, and  that $n_m=O(m^\beta)$ as $m \to \infty$, for some $\beta< \frac{1}{2}$.
\end{theorem}
\begin{proof}
We follow the argument of Theorem \ref{th:TLC para indicadores}, the case of sums of indicators. The (dependency) graph $\Gamma^{(n)}_m$ is the same except that the vertices are now labeled by the variables $\gcd(X^{(n)}_i,X^{(n)}_j)$. The parameters pertaining Janson's Theorem are now: number of vertices $N=\binom{m}{2}$, maximal degree $M=2(m-2)$, bound on the variables $A=n=n_m$, and
$$
\sigma^2=\binom{m}{2} \, \V\big(\gcd\big(X^{(n)}_1, X^{(n)}_2\big)\big)+m(m-1)(m-2) \, d^{(n)}_1\ge m(m-1)(m-2) \, d^{(n)}_1\, .
$$
Finally, for $h$ an integer so large that $\beta+1/{h} \le\frac{1}{2}$,
$$
\Big(\frac{N}{M}\Big)^{1/h} \ \frac{M\, A}{\sigma}\le \Big(\frac{\binom{m}{2}}{2(m-2)}\Big)^{1/h} \frac{2(m-2)n_m}{\sqrt{m(m-1)(m-2) \, d^{(n)}_1}}
 \underset{m \to \infty}{\longrightarrow}  0\, ,
$$
since $n_m=O(m^\beta)$, and since $d^{(n)}_1 \sim \Delta_{\text{\rm Toth}} \ln(n)^3$ (see Lemma \ref{lemma:asymptotics_d_n}).\end{proof}

\

\begin{remark}\upshape
It would be interesting to determine whether a restriction on the rate of growth of the sample space size like $n_m < m^{\beta}$, with $\beta < {1}/{2}$ which we have imposed  is necessary for the asymptotic normality of $\mathcal{Z}^{(n)}_m$, and if that is so, what is the optimal rate.
\end{remark}

\subsection{Extreme statistics of gcd of pairs}
\label{subsection:limit theorem for maxima}
We now turn our attention to the random variable which registers  the maximum of the greatest common divisors of pairs of the sample.
$$
\mathcal{M}^{(n)}_m=\max_{1\le i<j\le m} \{\gcd(X^{(n)}_i,X^{(n)}_j)\}.
$$
In \cite{DP2011}, Darling and Pyle studied the asymptotic behavior of the distribution of this variable, obtained some interesting results and asked   whether its normalized version
$$
\widetilde{\mathcal{M}}^{(n)}_m={\binom{m}{2}}^{-1} \,\mathcal{M}^{(n)}_m
$$
had a limit in distribution as $m\to\infty$ or not. The following theorem provides an answer
\begin{theorem}\label{th:distribution of max of gcd}
Let $m^\beta\le n \le e^{m^{\gamma}}$, for some $\beta>2$ and $\gamma <{1}/{3}$.
Then, for any $t>0$,
$$
\lim_{m\to\infty}\P\big(\widetilde{\mathcal{M}}^{(n)}_m \le t\big)= \exp\Big(-\frac{1}{t\zeta(2)}\Big).
$$
In other terms, $\widetilde{\mathcal{M}}^{(n)}_m$ tends, in distribution,  as $m\to\infty$, to the Fr\'{e}chet distribution with shape parameter~$1$ and scale parameter $1/\zeta(2)$.
\end{theorem}

Observe that this convergence result requires that the size of the sampling space $n$ tends to infinity along with $m$, the sample size, in contrast to the asymptotic normality results for the variables $\mathcal{C}^{(n)}_m$ and $\mathcal{Z}^{(n)}_m$, where the size of the sampling space $n$ played a relatively secondary role (see Sections \ref{subsection:distribucion indicadores}~and~\ref{subseccion:gcd of pairs}). Fr\'{e}chet distribution is one of the standard distributions used in Extreme Value Theory.

\smallskip

Theorem \ref{th:distribution of max of gcd} is a direct, and standard, consequence of the following result above Poisson convergence:
\begin{theorem}\label{th:poisson para gcd}
Let $n$ be as in Theorem {\upshape\ref{th:distribution of max of gcd}}. Let $t>0$ and consider the random variable
$$
N^{(n)}_m(t)=\#\Big\{1\le i<j\le m: \gcd(X^{(n)}_i,X^{(n)}_j)> t\binom{m}{2}\Big\}\, .
$$
Then, for each fixed $t >0$, the sequence $\{N^{(n)}_m(t)\}_m$ converges in distribution to a Poisson variable of parameter $\lambda=\frac{1}{t\zeta(2)}$:
$$
N^{(n)}_m(t)\overset{\text{d}}{\longrightarrow} \text{\rm Poisson}\Big(\frac{1}{t\zeta(2)}\Big)\quad\text{as $m\to\infty$.}
$$
\end{theorem}
\begin{proof}[Proof of Theorem {\upshape\ref{th:distribution of max of gcd}}]  Just observe that according to Theorem \ref{th:poisson para gcd}
\begin{align*}
\P\big(\widetilde{\mathcal{M}}^{(n)}_m > t\big)&=\P\Big(\max_{1\le i<j\le m} \gcd(X^{(n)}_i,X^{(n)}_j)>t\binom{m}{2}\Big)
\\    &=
\P(N^{(n)}_m(t)>0)\xrightarrow[m\to\infty]{} 1-\exp\Big(-\frac{1}{\zeta(2)\, t}\Big)\, .
\end{align*}

\end{proof}

In the proof of Theorem \ref{th:poisson para gcd}, we will use results of Silverman and Brown (see Theorem~A in~\cite{SB1978}) and of Brown and Silverman (see Theorem~A in \cite{BS1979}) about Poisson convergence of $U$-statistics, which, for the pairwise case  we may write as follows:
\begin{theorem}[Brown--Silverman]\label{th:Silverman-Brown}
Let $Y_1, Y_2, \ldots, Y_M$ be iid random variables taking values on some space $\mathcal{S}$.  Let $g(x,y)$ be a symmetric function defined on $\mathcal{S}^2$ and taking values~$0$ and $1$. Denote by $T$ the counter
$$
T=\sum_{1 \le i <j \le M} g(Y_i, Y_j)
$$
Let $\lambda=\E(T)$ and
$$\rho= M^4 \, \text{\rm{cov}}\big(g\big(Y_1, Y_2\big), g\big(Y_2, Y_3\big)\big)\, .$$
Then
$$
\big|\P(T=k)-\P\big(\text{\rm{Poisson}}(\lambda)=k\big)\big|\le C\Big(\frac{\lambda^2}{M}+\sqrt{\frac{\rho}{M}}\Big)\, \quad \text{for each integer} \ k \ge 0\, ,
$$
where $C$ is some absolute constant.
\end{theorem}

\begin{proof}[Proof of Theorem {\upshape\ref{th:poisson para gcd}}]
a) We  shall require a simple estimate for the distribution function of the greatest common divisor of a random pair. The mass function of the $\gcd$ of a pair satisfies, (see, for instance, \cite{DE2004}), that,  for $1 \le j \le n$,
$$
\Big|\P\big(\gcd(X^{(n)}_1,X^{(n)}_2)=j\big)-\frac{1}{j^2\, \zeta(2)}\Big|\le 4\Big(\frac{1+\ln(n/j)}{nj}\Big)\le 4 \Big(\frac{1+\ln(n)}{n}\Big) \frac{1}{j}\,.
$$
We deduce that, for $0 \le k \le n$,
\begin{equation} \label{eq:estimacion distribucion de gcd}
\Big|\P(\gcd(X^{(n)}_1,X^{(n)}_2)>k)-\frac{1}{\zeta(2)}\sum_{j=k+1}^n \frac{1}{j^2}\Big|\le4\frac{\big(1+\ln(n)\big)^2}{n}\, ,
\end{equation}
b) We will also need a convenient estimate of $\E\big(\gcd(X^{(n)}_1,X^{(n)}_2)\cdot \gcd(X^{(n)}_2,X^{(n)}_3)\big)$. Recall  (see Lemma \ref{lemma:asymptotics_d_n}) that
$$
\text{\rm cov}(\gcd(X^{(n)}_1,X^{(n)}_2)\, , \,  \gcd(X^{(n)}_2,X^{(n)}_3)\big)=d^{(n)}_1\sim \Delta_{\rm Toth} \ln(n)^3\, ,
$$
and, consequently,
\begin{equation}\label{eq:bound_prod_gdcs}
\E\big(\gcd(X^{(n)}_1,X^{(n)}_2)\cdot \gcd(X^{(n)}_2,X^{(n)}_3)\big)=O \big(\ln(n)^3\big)\, .
\end{equation}

Consider a sequence $n=n^m$ satisfying the conditions $m^\beta\le n_m \le e^{m^{\gamma}}$, for some $\beta>2$ and $\gamma <{1}/{3}$. Fix $t>0$. To apply Theorem \ref{th:Silverman-Brown}, we define the function
$$
g_m(x,y)=\unogrande_{\gcd(x,y)>t \binom{m}{2}}\, .
$$
for $1\le x,y \le n$, and the random variable
$$
T_m=\sum_{1 \le i<j\le m} g_m\big(X^{(n)}_i, X^{(n)}_j\big)\, ,
$$
which counts the number of random pairs with $\gcd$ bigger than $t \binom{m}{2}$.

Let us estimate the corresponding parameters $\lambda_m$ and $\rho_m$.
First,
$$
\lambda_m=\E(T_m)=\binom{m}{2} \,\P\Big(\gcd\big(X^{(n)}_1, X^{(n)}_2\big)> t \,{\textstyle\binom{m}{2}}\Big)\, .
$$
We have that $\lim_{m \to \infty}\lambda_m=\frac{1}{t \zeta(2)}$. To verify this, let $K=\big\lfloor t \binom{m}{2}\big \rfloor$, and bound, using
the estimate~\eqref{eq:estimacion distribucion de gcd}:
$$
\lambda_m \le \binom{m}{2} \P\big(\gcd(X^{(n)}_1, X^{(n)}_2)> K\big) \le \binom{m}{2} \Big(\frac{1}{\zeta(2)}\sum_{j=K+1}^n \frac{1}{j^2}\Big)+\binom{m}{2} 4 \Big(\frac{(1+\ln(n))^2}{n}\Big)\, ,
$$
to deduce, since $n_m\ge m^\beta$, with $\beta >2$, that
$$
\limsup_{m \to \infty} \lambda_m\le \frac{1}{t\zeta(2)}\, .
$$
Using $K=\big\lceil t \binom{m}{2}\big \rceil$, one gets analogously that $\liminf_{m \to \infty}\lambda_m \ge \frac{1}{t\zeta(2)}$.

\smallskip

Next, using estimate  \eqref{eq:bound_prod_gdcs}, we may bound
$$
\begin{aligned}
\rho_m&=m^4 \, \text{\rm cov}\big(\unogrande_{\gcd(X^{(n)}_1, X^{(n)}_2)>t \binom{m}{2}}\, , \, \unogrande_{\gcd(X^{(n)}_2, X^{(n)}_3)>t \binom{m}{2}}\big)
\\
&\le m^4 \, \E\big(\unogrande_{\gcd(X^{(n)}_1, X^{(n)}_2)>t \binom{m}{2}}\, \cdot \, \unogrande_{\gcd(X^{(n)}_2, X^{(n)}_3)>t \binom{m}{2}}\big)\\
&\le m^4 \, \frac{1}{t^2 \binom{m}{2}^2}\,\E\big(\gcd(X^{(n)}_1, X^{(n)}_2)\, \cdot \, \gcd(X^{(n)}_2, X^{(n)}_3)\big)=\frac{1}{t^2} \,O(\ln(n)^3)\, .\end{aligned}
$$

Reverting to the notation of the statement of the theorem, and on account of Theorem~\ref{th:Silverman-Brown} of Brown and Silverman,  this estimate of $\rho_m$ implies that
$$
\big|\P\big(N^{(n)}_m(t)=k\big)-\P\big(\text{\rm Poisson}(\lambda_m)=k\big)\big|\le C^{\prime} \Big(\frac{\lambda_m^2}{m}+\frac{\ln(n)^{3/2}}{t\sqrt{m}}\Big)\le C'' \Big(\frac{\lambda_m^2}{m}+\frac{m^{3\gamma/2}}{t\sqrt{m}}\Big)\,,
$$
since $n=n_m\le e^{m^\gamma}$. Finally, since $\gamma <1/3$,
this gives that
$$
\lim_{m \to \infty}\P\big(N^{(n)}_m(t)=k\big)=\P\big(\text{\rm Poisson}\big({\textstyle\frac{1}{t \zeta(2)}}\big)=k\big)
$$
for any integer $k \ge 0$, as desired.
\end{proof}

\begin{remark}
{\upshape
%Observe that the argument of the proof of Theorem \ref{th:poisson para gcd} gives that the result holds for any sequence $n_m$ as long as $n_m \ge m^2 \ln(n_m)^\alpha$, for some $\alpha >2$ and $n_m \le e^{m^\gamma}$, for some $\gamma < \frac{1}{3}$.
About the lower restriction on $n_m$ in Theorem \ref{th:poisson para gcd} there is not much to say, since just the statement of convergence requires that
$n_m/\binom{m}{2} \to +\infty$,  but it would be nice to know what is the upper restriction required, if any.}
\end{remark}

From Theorem \ref{th:distribution of max of gcd}, we deduce as a corollary an asymptotic concentration result of Darling and Pyle, \cite{DP2011}, Theorem 1:
\begin{corollary}\label{cor:darling_pyle}If\/ $n=n_m$ satisfies $n \ge m^\beta$, for some $\beta >2$ and, also, $n \le e^{m^\gamma}$, for some $\gamma<{1}/{3}$, then,  for any sequence $\delta_m>0$ with $\lim_{m\to \infty}\delta_m=0$, we have that
$$
\lim_{m \to \infty} \P\Big(m^{2}\delta_m< \max_{1\le i<j \le m} \gcd\big(X^{(n)}_i, X^{(n)}_j\big)  < m^{2} \frac{1}{\delta_m}\Big)=1\, .
$$
\end{corollary}

\begin{remark}\upshape
Notice that Darling and Pyle prove the above corollary  for the sequence $n_m=e ^{\alpha m}$, where $\alpha$ is any positive number; a sequence which is beyond the range of our Corollary~\ref{cor:darling_pyle}. It would be interesting to determine the optimal rate of growth of~$n_m$ for the validity both of Theorem~\ref{th:distribution of max of gcd} and of Corollary~\ref{cor:darling_pyle}.
\end{remark}

\section{$U$-statistics for greatest common divisors of $r$-tuples}\label{section:statistics_r_tuples}

We shall assume throughout this section that $r \ge 3$. We consider now $U$-statistics summing over the collection of subsets of size $r$ of the random sample of length $m$.

\subsection{Number of relatively prime $r$-tuples}\label{subsubsection:distribucion indicadores r a r}
Let us start with the variable
$$
\mathcal{C}^{(n)}_{m,r}=\sum_{1\le i_1<\cdots <i_r\le m} \text{\large\bf 1}_{\gcd(X_{i_1},\dots, X_{i_r})=1},
$$
the sum of $\binom{m}{r}$ terms counting the number of coprime $r$-tuples in a random sample of size~$m$ drawn uniformly from $\{1, \ldots, n\}$.

We have:
\begin{theorem}\label{th:TLC para indicadores r a r} For fixed $r \ge 3$ and for any sequence $n_m \ge 2$,
$$
\frac{\mathcal{C}^{(n_m)}_{m,r}-\E(\mathcal{C}^{(n_m)}_{m,r})}{\sqrt{\V(\mathcal{C}^{(n)}_{m,r})}}
\ {\overset{\text{d}}{\longrightarrow}}\ \mathcal{N}\quad\text{as $m \to \infty$.}
$$
\end{theorem}

The argument to prove Theorem~\ref{th:TLC para indicadores r a r} follows the same steps as the case of pairs; so that we shall only indicate some specific differences. The mean of $\mathcal{C}^{(n)}_{m,r}$ is given by
$$
\E(\mathcal{C}^{(n)}_{m,r})=\binom{m}{r}\, \P\big(\gcd(X^{(n)}_{1},\dots, X^{(n)}_{r})=1\big)=\binom{m}{r} \mu^{(n)}_{r-1}\, ;
$$
recall, from Lemma \ref{lemma:limit_mean_marginals}, that $\lim_{n \to \infty} \mu^{(n)}_{r-1}=\frac{1}{\zeta(r)}$.

%In the estimate of the variance, the following basic combinatorial lemma will be useful:
%\begin{lemma}
%\label{lema:varianzas de r en r}The number of pairs $(A,B)$, where $A,B\subset\{1,\dots,m\}$, $|A|=|B|=r$ is
%$$
%\binom{m}{r}^2=\sum_{s=0}^r {m\choose s} \,{{m-s}\choose {r-s}}\, {{m-r}\choose {r-s}}.
%$$
%In the right hand term, we have classified the pairs $(A,B)$ according to the size of the intersection $|A\cap B|=s$, where $0\le s\le r <m$.
%\end{lemma}

To estimate the variance of $\mathcal{C}^{(n)}_{m,r}$ we now follow standard manipulations of $U$-statistics. We need to consider some more covariances. Let us define, for $0 \le s\le r$,
\begin{equation}
\label{eq:covarianzas indicadores r a r}
\gamma^{(n)}_{r,s}=\text{cov}\big(\text{\large\bf 1}_{\gcd(X^{(n)}_{1},\dots, X^{(n)}_{s}, X^{(n)}_{s+1},\dots, X^{(n)}_r)=1},
\text{\large\bf 1}_{\gcd(X^{(n)}_{1},\dots, X^{(n)}_{s}, X^{(n)}_{r+1},\dots, X^{(n)}_{2r-s})=1}\big)\,.
\end{equation}
Observe that the two indicator functions involved in $\gamma^{(n)}_{r,s}$ have exactly $s$ of the variables $X^{(n)}_j$ in common. Notice that $\gamma^{(n)}_{r,1}=c^{(n)}_{r-1}$, see equation \eqref{eq:cnr_as_covariance}, and that $\gamma^{(n)}_{r,0}=0$, because of the independence of the $X^{(n)}_j$'s. Observe that, from the Cauchy--Schwarz inequality,
\begin{equation}
\gamma_{r,s}^{(n)}\le \gamma_{r,r}^{(n)}=\V\big(\text{\large\bf 1}_{\gcd(X^{(n)}_{1},X^{(n)}_{2},\ldots,  X^{(n)}_r)=1}\big)
\end{equation}
for each $0\le s\le r$. In fact (see, for instance, \cite{Serfling}, p.\ 182), $\gamma_{r,s}^{(n)}$ increases with $s$, and, in particular, $\gamma_{r,s}^{(n)} \ge 0$, for $0 \le s \le r$.

\smallskip
In terms of these covariances, the variance of $\mathcal{C}^{(n)}_{m,r}$ may be written  as
\begin{equation}
\label{eq:variance_sum_indicators_r_a_r}
\V(\mathcal{C}^{(n)}_{m,r})=\sum_{s=0}^r \binom{m}{s} \,\binom{m-s}{r-s}\, \binom{m-r} {r-s}\, \gamma_{r,s}^{(n)}.
\end{equation}
The product of binomial coefficients in the summand of index $s$ of this expression counts the number of pairs of subsets of size $r$ with intersection of size $s$ drawn from $\{1, 2, \ldots, m\}$. Observe that, with $s,r$ fixed and as $m \to \infty$,
$$
\binom{m}{s} \,\binom{m-s}{r-s}\, \binom{m-r}{r-s} \sim \frac{1}{s!\, (r-s)!}\  m^{2r-s}.
$$

We may trivially bound (just keeping the term $s=1$ in \eqref{eq:variance_sum_indicators_r_a_r} and using that $\gamma_{r,s}^{(n)}\ge 0$),
\begin{equation}\label{eq:asintotico de la varianza indicadores r a r}
\V(\mathcal{C}^{(n)}_{m,r}) \ge m \binom{m-1}{r-1}\binom{m-r}{r-1} c^{(n)}_{r-1}\, .
\end{equation}
Recall, see Lemma \ref{lemma:asymptotics_c_n},  that $\lim_{n \to \infty}c^{(n)}_{r-1}=S^{(r-1)}_2-(S^{(r-1)}_1)^2 $, a positive quantity.

%As $m\to\infty$, using that $\gamma_{r,0}(n)=0$, we obtain that for certain constants $C_{r,s}>0$,
%$$
%\V_n(\mathcal{C}_{m,r})\sim  \sum_{s=0}^r C_{r,s} m^{2r-s}\, \gamma_{r,s}^{(n)}= C_{r,1}\, m^{2r-1}\, \gamma_{r,1}(n)+C_{r,2}\, m^{2r-2}\, \gamma_{r,2}(n)+\cdots ,
%$$
%so that we can deduce that
%\begin{equation}\label{eq:asintotico de la varianza indicadores r a r}
%\V_n(\mathcal{C}_{m,r})\sim C\, m^{2r-1}
%\end{equation}
%as $m\to\infty$ if we prove that the covariance $\gamma_{r,1}(n)>0$.

\begin{proof}
[Proof of Theorem {\upshape\ref{th:TLC para indicadores r a r}}] We shall apply again Janson's Theorem \ref{th:Janson}. Consider the (dependency) graph with $\binom{m}{r}$ vertices labeled with the variables $\text{\large\bf 1}_{\gcd(X_{i_1},X_{i_2},\dots, X_{i_r})=1}$ for $1 \le i_1<i_2<\cdots < i_r\le n$, and with an edge joining two vertices if they have \textit{at least} one index of their labels in common. We record now the appropriate parameters in order to apply Theorem \ref{th:Janson}: the number of vertices $N=\binom{m}{r}$; the bound on the variables, $A=1$, since the variables  are just indicators; the maximal degree
$$
M=\binom{m}{r}- \binom{m-r}{r}-1 \sim \frac{r}{(r-1)!}\, m^{r-1}, \, \quad \text{as} \ m \to \infty\, ,
$$
and
$$
\sigma^2 \ge m \binom{m-1}{r-1}\binom{m-r}{r-1} c^{(n)}_{r-1}\, .
$$

Fix any integer $h \ge 3$. For some constant $C_{r,h}$, we have that
$$
\Big(\frac{N}{M}\Big)^{1/h} \ \frac{M\,A}{\sigma}\le C_{r,h} \Big(\frac{m^r}{m^{r-1}}\Big)^{1/h} \frac{m^{r-1}}{m^{r-1/2}}\frac{1}{\sqrt{c^{(n)}_{r-1}}}=C_{r,h} \frac{m^{1/h}}{m^{1/2}}\frac{1}{\sqrt{c^{(n)}_{r-1}}}\,,
$$
which converges to $0$ as $m \to \infty$, whatever the sequence $n_m \ge 2$.\end{proof}

\subsection{Sums of greatest common divisors of $r$-tuples}\label{subseccion:gcd of pairs r a r}
For the variable
$$
\mathcal{Z}^{(n)}_{m,r}=\sum_{1\le i_1<\cdots <i_r\le m} \gcd(X^{(n)}_{i_1},\dots, X^{(n)}_{i_r})
$$
which sums the greatest common divisors of all the $r$-tuples of a random sample of length~$m$ drawn for $\{1, \ldots, n\}$, we have:
\begin{theorem}\label{th:TLC para gcd r a r}
For fixed $r \ge 3$ and for any sequence $n_m$ of integers satisfying $2 \le n_m \le m^{\beta}$, for some $\beta <1/2$,
$$
\frac{\mathcal{Z}^{(n_m)}_{m,r}-\E(\mathcal{Z}^{(n_m)}_{m,r})}{\sqrt{\V(\mathcal{Z}^{(n_m)}_{m,r})}}\ {\overset{\text{d}}{\longrightarrow}} \ \mathcal{N}\quad\text{as $m\to\infty$}.
$$
\end{theorem}

The proof of Theorem \ref{th:TLC para gcd r a r} is a variation of the proof of Theorem \ref{th:TLC para indicadores r a r}. We just discuss a few of ingredients.

The mean of  $\mathcal{Z}^{(n)}_{m,r}$ is given by
$$
\E(\mathcal{Z}^{(n)}_{m,r})=\binom{m}{r}\, \E(\gcd(X^{(n)}_{1},\dots, X^{(n)}_{r}))=\binom{m}{r}\, \nu^{(n)}_{r-1}
$$
Let us define, for $0 \le s\le r$,
\begin{equation}
\label{eq:covarianzas gcd r a r}
\omega^{(n)}_{r,s}=\text{cov} (\gcd(X^{(n)}_{1},\dots, X^{(n)}_{s}, X^{(n)}_{s+1},\dots, X^{(n)}_r),
\gcd(X^{(n)}_{1},\dots, X^{(n)}_{s}, X^{(n)}_{r+1},\dots, X^{(n)}_{2r-s})).
\end{equation}
Notice that $\omega^{(n)}_{r,1}=d^{(n)}_{r-1}$, see \eqref{eq:dnr_as_covariance}. Again, $\omega^{(n)}_{r,0}=0$, because of the independence of the $X^{(n)}_j$'s. Again, from Cauchy--Schwarz,
\begin{equation}
\omega_{r,s}^{(n)}\le \omega_{r,r}^{(n)}=\V\big(\gcd(X^{(n)}_{1},X^{(n)}_{2},\ldots,  X^{(n)}_r)\big)
\end{equation}
for each $0\le s\le r$. And again, $0=\omega_{r,0}^{(n)}\le \omega_{r,s}^{(n)}\le \omega_{r,r}^{(n)}$ for $0\le s\le r$.

\smallskip
The variance of $\mathcal{Z}^{(n)}_{m,r}$ may be written as
$$
\V(\mathcal{Z}^{(n)}_{m,r})=\sum_{s=0}^r \binom{m}{s} \,\binom{m-s}{r-s}\, \binom{m-r}{r-s}\, \omega^{(n)}_{r,s}(n)\, ,
$$
so that we may bound
$$
\V(\mathcal{Z}^{(n)}_{m,r}) \ge m \binom{m-1}{r-1}\binom{m-r}{r-1} d^{(n)}_{r-1}\, .
$$
Recall, see Lemma \ref{lemma:asymptotics_d_n}, that, for $r \ge 3$,
$$
\lim_{n \to \infty} d^{(n)}_{r-1}=\sum_{1 \le i,j < \infty} \frac{\varphi(i)}{i^{r}}\,\frac{\varphi(j)}{j^{r}}\,\big(\gcd(i,j)-1\big)\, ,
$$
which is a positive and finite (since $r\ge 3$) quantity.

For the proof of Theorem \ref{th:TLC para gcd r a r}, we just have to observe that the parameter $A$ to apply in Janson's Theorem is now $A=n$, and this is why we require now the bound $n_m\le m^{\beta}$, with~$\beta < {1}/{2}$.

\begin{remark}[Extreme Statistics of the greatest common divisor of $r$-tuples]\label{remark:limit theorem for maxima r a r}\upshape
 Fix $r\ge 3$.
 It would be interesting to determine, if there is any at all, the corresponding approximation result for the maximum of $\gcd$ for $r$-tuples.

 Let us see why the approach which we have followed for the case of pairs breaks down for $r \ge 3$. Following that approach, one would fix $t>0$, consider the counter
$$
T_m=\sum_{1\le i_1<i_2 < \ldots 1_r \le n} \unogrande_{\gcd(X^{(n)}_{i_1},X^{(n)}_{i_2},\ldots, X^{(n)}_{i_r})>t s_m}\,,
$$
where $\{s_m\}_m$ is some appropriate sequence, and expect to obtain convergence in distribution of $T_m$ to a Poisson variable.

Now $\E(T_m)=\binom{m}{r}\P\big(\gcd(X^{(n)}_1, X^{(n)}_2, \ldots, X^{(n)}_r) >t s_m\big)$ should converge to the parameter~$\lambda_t$ defining the purported limiting Poisson distribution.
The distribution of $\gcd$ of $r$-tuples satisfies, for $1 \le j \le n$, that
$$
\Big|\P\big(\gcd(X^{(n)}_1, X^{(n)}_2, \ldots, X^{(n)}_r) =j\big)-\frac{1}{\zeta(r)}\frac{1}{j^r}\Big|\le C_r \frac{1}{n j^{r-1}}\,,
$$
see, for instance, \cite{FF}, and therefore, for $1 \le k \le n$,
$$
\Big|\P\big(\gcd(X^{(n)}_1, X^{(n)}_2, \ldots, X^{(n)}_r) >k\big)-\frac{1}{\zeta(r)}\sum_{j=k+1}^n \frac{1}{j^r}\Big|\le C_r \frac{1}{n k^{r-2}}\,.
$$
With the forced choice of $s_m=\binom{m}{r}^{{1}/{(r-1)}}$, and as long as $\frac{n}{m^{r/(r-1)}} \to \infty$, we have that
$$
\binom{m}{r}\P\big(\gcd(X^{(n)}_1, X^{(n)}_2, \ldots, X^{(n)}_r) >t s_m\big)\to \frac{1}{t^{r-1} \zeta(r)}=\lambda_t\, .
$$

The general result of Brown and Silverman (Theorem $A$ of \cite{BS1979}) for Poisson convergence of $U$-statistics requires that
$$
m^{2r-1} \text{cov} \big(\unogrande_{\gcd(X^{(n)}_1, X^{(n)}_2, \ldots, X^{(n)}_{r-1}, X^{(n)}_{r+1})> t s_m},
\unogrande_{\gcd(X^{(n)}_1, X^{(n)}_2, \ldots, X^{(n)}_{r-1}, X^{(n)}_{r})>t s_m}\big) \to 0\, ,
$$
as $m \to \infty$.

If we simply estimate, as we did in the case of pairs,
\begin{align*}
&\text{\rm cov} \big(\unogrande_{\gcd(X^{(n)}_1, X^{(n)}_2, \ldots, X^{(n)}_{r-1}, X^{(n)}_{r+1})> t s_m},
\unogrande_{\gcd(X^{(n)}_1, X^{(n)}_2, \ldots, X^{(n)}_{r-1}, X^{(n)}_{r})>t s_m}\big)
\\
&\quad\le
\P\big(\gcd(X^{(n)}_1, X^{(n)}_2, \ldots, X^{(n)}_{r-1}, X^{(n)}_{r+1})> t s_m,
\gcd(X^{(n)}_1, X^{(n)}_2, \ldots, X^{(n)}_{r-1}, X^{(n)}_{r})>t s_m\big)
\\
&\quad\le  \frac{1}{t^2 s_m^2} \,\E\big(\gcd(X^{(n)}_1, X^{(n)}_2, \ldots, X^{(n)}_{r-1}, X^{(n)}_{r+1})\cdot
\gcd(X^{(n)}_1, X^{(n)}_2, \ldots, X^{(n)}_{r-1}, X^{(n)}_{r})\big)\, .
\end{align*}
we get nowhere, because the expectation above is obviously at least 1, and
$$
\frac{m^{{2r-1}}}{s_m^2}\asymp \frac{m^{{2r-1}}}{m^{2\frac{r}{r-1}}},
$$
which for $r=2$ tends to $0$ with $m$, but for $r\ge3$, our present case, tends to $\infty$ with $m$.

\smallskip

To obtain an asymptotic approximation results for the maximum of $\gcd$ for $r$-tuples following the approach which we have followed one would need a better estimate, if possible,~of
$$
\text{\rm cov} \big(\unogrande_{\gcd(X^{(n)}_1, X^{(n)}_2, \ldots, X^{(n)}_{r-1}, X^{(n)}_{r+1})> t s_m},
\unogrande_{\gcd(X^{(n)}_1, X^{(n)}_2, \ldots, X^{(n)}_{r-1}, X^{(n)}_{r})>t s_m}\big)  \, .
$$
\end{remark}

\section{Higher moments}\label{section:statistics_higher_moments}

Finally, we consider in this section $U$-statistics of \textit{moments}, other than first,  of  $\gcd$.
We follow, of course,  the general approach of previous sections,  particularly, Section \ref{subseccion:gcd of pairs r a r}; we will just mention the few extra ingredientes needed to obtain the corresponding results for higher moments.

We fix throughout this section the \textit{integer} exponent $q \ge 1$ and the length $r$ for the evaluation of $\gcd$'s, and consider
$$
\mathcal{Z}^{(n)}_m=\sum_{1 \le i_1<i_2<\cdots <i_r\le m}\gcd\big(X^{(n)}_{i_1}, X^{(n)}_{i_2}, \ldots, X^{(n)}_{i_r}\big)^q\, .
$$
Observe that, departing from previous usage, we are not decorating $\mathcal{Z}^{(n)}_m$ with the length $r$ (or the exponent $q$).

\smallskip

For $n$ fixed, $n\ge 2$, and as $m \to \infty$, we have asymptotic normality for $\mathcal{Z}^{(n)}_m$. This follows exactly as in Section \ref{subseccion:gcd of pairs r a r}.

\begin{theorem}
Given a length $r \ge 2$ and an  exponent $q\ge 1$, then for fixed $n \ge 2$
$$
\mathcal{Z}^{(n)}_m=\sum_{1 \le i_1<\cdots <i_r\le m}\gcd\big(X^{(n)}_{i_1},  \ldots, X^{(n)}_{i_r}\big)^q \quad \text{is asymptotically normal as} \ m \to \infty\, .
$$
\end{theorem}

For varying $n=n_m$, the general approach hinges on estimating (from below) the covariance
$$
\omega^{(n)}=\text{\rm cov}\big(\gcd\big(X^{(n)}_1, X^{(n)}_2, \ldots, X^{(n)}_r\big)^q, \gcd\big(X^{(n)}_1, X^{(n)}_{r+1}, \ldots, X^{(n)}_{2r-1}\big)^q\big)
$$
(just one variable $X^{(n)}$ in common). Now, as $n \to \infty$,
$$
\E\big(\gcd\big(X^{(n)}_1, X^{(n)}_2, \ldots, X^{(n)}_r\big)^q\big) \quad\begin{cases}
\longrightarrow  {\zeta(r-q)}/{\zeta(r)}\, , & \ \text{if} \ q \le r-2\, , \\[4pt]
\sim  {\ln(n)}/{\zeta(r)}\, , & \ \text{if} \ q =r-1\, , \\[4pt]
\sim D_{r,q} \cdot n^{q-r+1} \, , & \ \text{if} \ q \ge r\,,
\end{cases}
$$
for some constant $D_{r,q}$ (see \cite{FF}).

Conditioning on the value of $X^{(n)}_1$ and using Ces\`aro's marginal formula \eqref{eq:cesaro_marginal_formula}, we may write
\begin{align*}
\pi^{(n)}&:=\E\big(\gcd\big(X^{(n)}_1, X^{(n)}_2, \ldots, X^{(n)}_r\big)^q \, \cdot\,  \gcd\big(X^{(n)}_1, X^{(n)}_{r+1}, \ldots, X^{(n)}_{2r-1}\big)^q\big)
\\
&=\frac{1}{n} \sum_{k=1}^n \Big(\frac{1}{n^{r-1}} \sum_{j|k} \varphi_q(j)\nfracj^{r-1}\Big)^2
\end{align*}
We split the analysis of $\pi^{(n)}$ into the three cases above.

\medskip

a) For $q \le r-2$, we first write
\begin{align*}
\pi^{(n)}&=\frac{1}{n}\sum_{k=1}^n \frac{1}{n^{2(r-1)}}\sum_{i,j |k} \varphi_q(i)\varphi_q(j) \Big\lfloor\frac{n}{i}\Big\rfloor^{r-1}\Big\lfloor\frac{n}{j}\Big\rfloor^{r-1}\\
&=\sum_{i,j \le n} \varphi_q(i)\varphi_q(j) \Big\lfloor\frac{n}{i}\Big\rfloor^{r-1}\Big\lfloor\frac{n}{j}\Big\rfloor^{r-1}\Big\lfloor\frac{n}{\lcm(i,j)}\Big\rfloor\frac{1}{n^{2(r-1)}}\, ,
\end{align*}
and then bound
\begin{equation*}
\pi^{(n)}\le \sum_{i,j \le n} \frac{ \varphi_q(i)}{i^r}\frac{ \varphi_q(i)}{i^r}\gcd(i,j)\,.
\end{equation*}
%
%\begin{align*}
%\pi^{(n)} \le \frac{1}{n} \sum_{k=1}^n \Big( \sum_{j|k} \frac{\varphi_q(j)}{j^{r-1}}\Big)^2
%&=\sum_{i,j \le n}  \frac{\varphi_q(i)}{i^{r-1}} \frac{\varphi_q(j)}{j^{r-1}} \Big(\frac{1}{n} \Big\lfloor \frac{n}{\text{\rm lcm}(i,j)}\Big\rfloor\Big)\\
%&\le\sum_{i,j \le n}  \frac{\varphi_q(i)}{i^{r}} \frac{\varphi_q(j)}{j^{r}} \gcd(i,j)\, .
%\end{align*}

Since
$$
\sum_{i,j \le n}  \frac{\varphi_q(i)}{i^{r}} \frac{\varphi_q(j)}{j^{r}}  \gcd(i,j)\le
\sum_{i,j \le n}  \frac{\gcd(i,j)}{i^{r-q}j^{r-q}}\le \sum_{i,j\ge 1}  \frac{\gcd(i,j)}{i^{r-q}j^{r-q}}=\frac{\zeta(r-q)^2 \zeta(2(r-q)-1)^2}{\zeta{(2(r-q))}}
$$
(see \eqref{eq:sum gcd and zeta}), we may conclude from dominated convergence that
$$
\lim_{n \to \infty} \pi^{(n)}=\sum_{i,j \le n} \frac{\varphi_q(i)}{i^{r}} \frac{\varphi_q(j)}{j^{r}}  \gcd(i,j)\, .
$$

Also,
$$
\lim_{n \to \infty}\E\big(\gcd\big(X^{(n)}_1, X^{(n)}_2, \ldots, X^{(n)}_r\big)^q\big)=\dfrac{\zeta(r-q)}{\zeta(r)}=\sum_{j=1}^{\infty} \frac{\varphi_q(j)}{j^r}\,
$$
so that, finally, we have, in this case, $q \le r-2$, that
$$
\lim_{n \to \infty} \omega^{(n)}=\sum_{i,j=1}^{\infty} \frac{\varphi_q(i)}{i^{r}} \,\frac{\varphi_q(j)}{j^{r}} \,\big( \gcd(i,j)-1\big)\,
$$

This means that we have asymptotic normality as long as $n_m^q \le m^{\beta}$ for some $\beta < {1}/{2}$.

\medskip

b) Case $q=r-1$. We shall get that $\pi^{(n)}$ is at least of order  $\ln(n)^3$. To see this use (twice) that $\lfloor x \rfloor \ge x/2$, if $\lfloor x \rfloor \ge 1$, to bound $\pi^{(n)}$ from below:
\begin{align*}
\pi^{(n)} \ge \frac{1}{4} \frac{1}{n} \sum_{k=1}^n \Big( \sum_{j|k} \frac{\varphi_q(j)}{j^{q}}\Big)^2
&=\frac{1}{4}\sum_{i,j \le n}  \frac{\varphi_q(i)}{i^{q}} \frac{\varphi_q(j)}{j^{q}} \Big(\frac{1}{n} \Big\lfloor \frac{n}{\text{\rm lcm}(i,j)}\Big\rfloor\Big)\\
&\ge\frac{1}{8}\sum_{\lcm(i,j) \le n}  \frac{\varphi_q(i)}{i^{q+1}} \frac{\varphi_q(j)}{j^{q+1}} \gcd(i,j)
\, .
\end{align*}
Appealing now to Corollary \ref{cor:partial_sums_identity_double_sum_varphi_general_gcd}, we deduce that
$$
\liminf_{n \to \infty} \frac{\pi^{(n)} }{\ln(n)^3} \ge \frac{1}{8} \Delta_q\, ,
$$
and further that
$$
\liminf_{n \to \infty}\frac{\omega^{(n)} }{\ln(n)^3} \ge \frac{1}{8} \Delta_q\, ,
$$
since
$$
\pi^{(n)}-\omega^{(n)} \sim \Big(\frac{1}{\zeta(2)}\Big)^2 \ln(n)^2\,.
$$

The outcome of all this is again  that we have asymptotic normality for $\mathcal{Z}^{(n)}_m$ as long as $n_m^q \le m^{\beta}$ for some $\beta < {1}/{2}$.

\medskip

We stop and record the consequence of the analysis in these two cases a) and b).

\begin{theorem}
Given a length $r \ge 2$ and a exponent $q\ge 1$ with $q \le r-1$, then for any sequence $n_m$ satisfying $2 \le n_m$ and $n_m^q \le m^\beta$, with $\beta <{1}/{2}$,
$$
\mathcal{Z}^{(n)}_m=\sum_{1 \le i_1<\ldots <i_r\le m}\gcd\big(X^{(n)}_{i_1},  \ldots, X^{(n)}_{i_r}\big)^q \quad \text{is asymptotically normal as} \ m \to \infty\, .
$$
\end{theorem}

c) Case $q \ge r$. One would expect that both $\pi^{(n)}$ and $\omega^{(n)}$ would grow in this case as $n^{2(q-r+1)}$. But \textit{we have not been able to ascertain that}. Nonetheless, if that were the case, then one would have asymptotic normality as long as $2 \le n_m$ and $n_m^{r-1} \le m^\beta$ with $\beta <{1}/{2}$.

\section{Strong law}\label{section:strong_law}

The sequence of counters  $\mathcal{C}^{(n)}_{m,r}$ indexed by $m$, with sample space $\{1, \ldots, n\}$,  $n\ge 2$ fixed, and length $r \ge 2$ fixed, do satisfy a strong law of large numbers as $m \to \infty$.

\begin{theorem}\label{theorem:strong_law_indicators}
$$
\lim_{m\to \infty}\frac{\mathcal{C}^{(n)}_{m,r}}{\E(\mathcal{C}^{(n)}_{m,r})}= 1 \quad \text{almost surely}\, .
$$
\end{theorem}
In other terms, for almost all realizations of the complete sequence  $x^{(n)}_1, x^{(n)}_2, \ldots$, the sequence $\Big\{\frac{\mathcal{C}^{(n)}_{m,r}\big(x^{(n)}_1, x^{(n)}_2, \ldots, x^{(n)}_m\big)}{\E(\mathcal{C}^{(n)}_{m,r})}\Big\}_m$, where each successive term is calculated using the values of the given realization $\big\{x^{(n)}_k\big\}_k$, converges to 1.

Since $\E(\mathcal{C}^{(n)}_{m,r})=\binom{m}{r}\mu^{(n)}_{r-1}$, we could also write
$$
\lim_{m\to \infty}\frac{\mathcal{C}^{(n)}_{m,r}}{\binom{m}{r}}= \mu^{(n)}_{r-1}, \quad \text{almost surely}\, .
$$
Recall that, as $n\to \infty$, the mean  $\mu^{(n)}_{r-1}$ converges to $1/\zeta(r)$.

There are general strong laws for $U$-statistics which could be applied, but we prefer, given our previous estimates of variances and covariances of $\gcd$'s,  to derive Theorem~\ref{theorem:strong_law_indicators} directly from the following  (standard) lemma:
\begin{lemma}\label{lemma:strong_law}
Let $(Y_m)_m$ be an increasing sequence of positive random variables in a probability space, such that
\begin{enumerate}
\item[\rm 1)] $\E(Y_m)$ increases to infinity at a polynomial rate,
\item[\rm 2)] $\V(Y_m) \le C \ \E(Y_m)^\delta$, for some $0<\delta <2$.
\end{enumerate}Then
$$
\lim_{m \to \infty} \frac{Y_m}{\E(Y_m)}=1 \quad \text{almost surely}\, .
$$
\end{lemma}
By ``increasing at a polynomial rate'' we mean that $\E(Y_m) \sim C m^\beta$, for some $\beta >0$, as $m \to \infty$.

\begin{proof}[Proof of Theorem {\upshape\ref{theorem:strong_law_indicators}}]
Fix $n \ge 2$ and $r \ge 2$. Let us verify that $Y_m=\mathcal{C}^{(n)}_{m,r}$ satisfies the hypothesis of Lemma \ref{lemma:strong_law}
Obviously $0 \le Y_m \le Y_{m+1}$. Besides, $\E(Y_m)=\E(\mathcal{C}^{(n)}_{m,r})=\binom{m}{r}\mu^{(n)}_{r-1}$ grows at polynomial rate, with $\beta=r$.

Recall, from \eqref{eq:variance_sum_indicators_r_a_r}, that
$$
\V(\mathcal{C}^{(n)}_{m,r})=\sum_{s=0}^r \binom{m}{s} \,\binom{m-s}{r-s}\, \binom{m-r} {r-s}\, \gamma_{r,s}^{(n)}\,.
$$
Now,  since $\gamma_{r,s}^{(n)}\le \omega_{r,r}^{(n)}$, for $s$ from $s=0$ to $s=r$, and taking into account that $\gamma_{r,0}(n)=0$, we may bound
$$\begin{aligned}
\V(\mathcal{C}^{(n)}_{m,r})&\le \omega_{r,r}^{(n)} \Big[ {m\choose r}^2-\binom{m}{r}\binom{m-r}{r}\Big]\\
&\le\omega_{r,r}^{(n)}\binom{m}{r} \Big[ \binom{m}{r}-\binom{m-r}{r}\Big] \le C_r \,\omega_{r,r}^{(n)}     \,m^{2r-1}\, .
\end{aligned}
$$
Since $\E(\mathcal{C}^{(n)}_{m,r})=\binom{m}{r}\mu^{(n)}_{r-1}$, the second condition of Lemma \ref{lemma:strong_law} is satisfied, with $\delta=2-\frac{1}{r}$.
\end{proof}

For $\mathcal{Z}^{(n)}_{m,r}$ and even further for its $q$ moment version, there are analogous strong laws.

\newpage

For completeness, a proof of Lemma \ref{lemma:strong_law}  (modeled upon \cite{Durret}, Theorem 6.8) follows.
\begin{proof}[Proof of Lemma {\upshape\ref{lemma:strong_law}}]
Chebyshev's inequality gives
$$
\P\Big(\Big|\frac{Y_m}{\E(Y_m)}-1\Big|>\lambda\Big)\le \frac{1}{\lambda^2}\, \frac{\V(Y_m)}{\E(Y_m)^2}\le \frac{C}{\lambda^2} \frac{1}{\E(Y_m)^{2-\delta}}\,.
$$
This ensures that the subsequence
$$
\frac{Y_{m_k}}{\E(Y_{m_k})}\xrightarrow{\rm a.s.} 1 \quad\text{as $k\to\infty$,}
$$
if $m_k=\big\lfloor k^{\frac{2}{(2-\delta)\beta}}\big\rfloor$.
Now, for each $m$, such  $m_k \le m < m_{k+1}$,
$$
\frac{Y_m}{\E(Y_m)}\le \frac{Y_{m_{k+1}}}{\E(Y_{m_{k+1}})} \frac{\E(Y_{m_{k+1}})}{\E(Y_{m_{k}})}. $$
Since $m_{k+1}/m_k \to 1$ as $k \to \infty$, and because of the polynomial rate condition, we deduce that, almost surely,
$$
\limsup_{m \to \infty} \frac{Y_m}{\E(Y_m)}\le 1
$$
An analogous estimate from  below completes the proof.
\end{proof}

\

\noindent\textsc{Jos\'{e} L. Fern\'{a}ndez:} Departamento de Matem\'{a}ticas, Universidad Aut\'{o}noma de Madrid, 28049-Madrid, Spain.
\texttt{joseluis.fernandez@uam.es}

\medskip

\noindent\textsc{Pablo Fern\'{a}ndez:} Departamento de Matem\'{a}ticas, Universidad Aut\'{o}noma de Madrid, 28049-Madrid, Spain.
\texttt{pablo.fernandez@uam.es}

 \renewcommand{\thefootnote}{\fnsymbol{footnote}}
\footnotetext{The research of both authors is partially supported by Fundaci\'{o}n Akusmatika. The second named author is partially supported by the Spanish Ministerio de Ciencia e Innovaci\'{o}n, project no. MTM2011-22851.}
\renewcommand{\thefootnote}{\arabic{footnote}}

\end{document}